\documentclass[italian,UKenglish,a4paper]{amsart}
\usepackage{etex}

\usepackage[utf8]{inputenc}

\usepackage{babel}
\usepackage[a4paper,lmargin=2.0cm,rmargin=2.0cm,tmargin=4.0cm,bmargin=4.0cm]{geometry}

\usepackage{rotating,color,hyperref,appendix}
\usepackage{tocvsec2}
\usepackage{booktabs}
\usepackage{float}
\usepackage{tikz}
\usetikzlibrary{positioning,arrows,shapes,decorations.pathmorphing,shadows}

\usepackage{amsfonts}

\newcommand{\Spin}[1]{\ensuremath{\text{\upshape\rmfamily Spin}(#1)}}

\newcommand{\Sp}[1]{\mathrm{Sp}(#1)}
\newcommand{\SO}[1]{\mathrm{SO}(#1)}
\newcommand{\SU}[1]{\mathrm{SU}(#1)}
\newcommand{\U}[1]{\mathrm{U}(#1)}

\newcommand{\I}{\mathcal{I}}
\newcommand{\J}{J} 

\newcommand{\RO}{R}

\newcommand{\liespin}[1]{\mathop{\mathfrak{spin}}(#1)}
\newcommand{\liesu}[1]{\mathop{\mathfrak{su}}(#1)}

\newcommand{\lieso}[1]{\mathop{\mathfrak{so}}(#1)}

\newcommand{\Id}{\text{Id}}

\newcommand{\End}[1]{\mathrm{End}(#1)}
\newcommand{\Endcomplex}[1]{\mathrm{End}_\CC(#1)}
\newcommand{\Endminus}[1]{\mathrm{End}^-(#1)}
\newcommand{\Cl}{\mathrm{Cl}}

\newcommand{\PP}{\mathcal{P}}

\newcommand{\CC}{\mathbb{C}}
\newcommand{\HH}{\mathbb{H}}
\newcommand{\RR}{\mathbb{R}}
\newcommand{\ZZ}{\mathbb{Z}}

\newcommand{\OO}{\mathbb{O}}

\newcommand{\EIII}{\mathrm{EIII}}
\newcommand{\Esix}{\mathrm{E}_6}
\newcommand{\FII}{\mathrm{FII}}
\newcommand{\Ffour}{\mathrm{F}_4}
\newcommand{\tr}[1]{\mathop{\mathrm{tr}}(#1)}

\newcommand{\shortform}[1]{\ensuremath{{\scriptstyle\boldsymbol{#1}}}}

\numberwithin{equation}{section}

\theoremstyle{plain}
\newtheorem{te}{Theorem}[section]
\newtheorem*{te*}{Theorem}
\newtheorem{pr}[te]{Proposition}
\newtheorem{co}[te]{Corollary}

\theoremstyle{definition}
\newtheorem{de}[te]{Definition}

\theoremstyle{remark}
\newtheorem{re}[te]{Remark}

\date{\today}

\subjclass[2010]{Primary 53C26, 53C27, 53C38}

\begin{document}

\begin{otherlanguage}{italian}
\author{Maurizio Parton}
\address{Universit\`a di Chieti-Pescara\\ Dipartimento di Economia, viale della Pineta 4, I-65129 Pescara, Italy}
\email{parton@unich.it}
\author{Paolo Piccinni}
\address{Sapienza-Universit\`a di Roma\\ Dipartimento di Matematica, piazzale Aldo Moro 2, I-00185, Roma, Italy}
\email{piccinni@mat.uniroma1.it}
\end{otherlanguage}

\title{The even Clifford structure of the fourth Severi variety}

\begin{abstract}
The Hermitian symmetric space $M=\EIII$ appears in the classification of complete simply connected Riemannian manifolds carrying a parallel even Clifford structure \cite{MoSCSR}. This means the existence of a real oriented Euclidean vector bundle $E$ over it together with an algebra bundle morphism $\varphi:\Cl^0(E) \rightarrow \End{TM}$ mapping $\Lambda^2 E$ into skew-symmetric endomorphisms, and the existence of a metric connection on $E$ compatible with $\varphi$. We give an explicit description of such a vector bundle $E$ as a sub-bundle of $\End{TM}$. From this we construct a canonical differential 8-form on  $\EIII$, associated with its holonomy $\mathrm{Spin}(10)\cdot \mathrm{U}(1) \subset \mathrm{U}(16)$, that represents a generator of its cohomology ring. We relate it with a Schubert cycle structure by looking at $\EIII$ as the smooth projective variety $V_{(4)} \subset \CC P^{26}$ known as the fourth Severi variety.
\end{abstract}

\keywords{Clifford structure, exceptional symmetric space, octonions, canonical differential form}
\thanks{Both authors were supported by the GNSAGA group of INdAM and by the MIUR under the PRIN Project ``Variet\`a reali e complesse: geometria, topologia e analisi armonica''}

\maketitle
\tableofcontents

\section{Introduction}
This paper deals with the compact Hermitian symmetric space
\[
\EIII = \Esix/(\mathrm{Spin}(10)\cdot\mathrm{U}(1)).
\]
Its holonomy group $G=\mathrm{Spin}(10)\cdot \mathrm{U}(1) \subset \mathrm{U}(16)$ gives rise to a $G$-structure we will describe in details both in the flat space $\CC^{16}$ and in sixteen dimensional complex Hermitian manifolds.

The symmetric space $\EIII$ appears in the literature in more than one context. For example it is often called the \emph{projective plane over the complex octonions}. One can in fact construct $\EIII$ by starting from the complex exceptional Jordan algebra
\[
\mathcal H_3(\CC \otimes \OO) = \Bigg\{ \Bigg( \begin{matrix}
c_1&x_3&\bar x_2\\
\bar x_3&c_2&x_1\\
x_2&\bar x_1 &c_3
\end{matrix} \Bigg), \; c_i \in \CC, \; x_i \in \CC \otimes \OO \Bigg\} \cong \CC^{27}
\]
of $3 \times 3$ Hermitian matrices over the composition algebra $\CC \otimes_\RR \OO$ of complex octonions, whose product is defined as $(c_1 \otimes x_1)(c_2 \otimes x_2)= c_1c_2 \otimes x_1x_2$. The subgroup of $\mathrm{GL}(\mathcal H_3(\CC \otimes \OO)) = \mathrm{GL}(27, \CC)$ of complex linear transformations preserving
\[
\det A = \frac{1}{6}(\rm{trace} \; A)^3 -\frac{1}{2}(\rm{trace} \; A)(\rm{trace} \; A^2) +\frac{1}{3} \; trace \; A^3
\]
turns out to be the exceptional simple complex Lie group $\Esix(\CC)$.
The action of $\Esix(\CC)$ on the associated projective space $\CC P^{26}= P(\mathcal H_3(\CC \otimes \OO))$ has three orbits, defined by the possible values of the rank of matrices. The closed orbit, consisting of rank one matrices and defined by the quadratic equation
\[
A^2 = (\mathrm{trace} \; A) A,
\]
turns out to be the symmetric space $\EIII$. By using the compact subgroup $\Esix \subset \Esix(\CC)$ one gets as isotropy subgroup $\mathrm{Spin}(10)\cdot \mathrm{U}(1) = (\mathrm{Spin}(10)\times \mathrm{U}(1))/\ZZ_4 $. The outlined construction is parallel with that of the projective Cayley plane $\FII=\Ffour/\mathrm{Spin}(9)$, and this is one reason for naming $\EIII$ the \emph{projective plane over the complex octonions}. Cf.\ for example \cite{AtBPPS} for such constructions of $\FII$ and $\EIII$, and \cite[pages 86--90]{YokELG} for a careful description of the subgroup $\mathrm{Spin}(10)\cdot \mathrm{U}(1)\subset \Esix$.

Our motivation for the present work has been the study of $\Spin{9}$-structures on 16-dimensional Riemannian manifolds. We did this in our previous work following the approach of Th.\ Friedrich \cite{FriWSS}, i.e.\ via a suitably chosen rank 9 vector sub-bundle of the endomorphism bundle (cf.\ also Section \ref{preliminaries} as well as \cite{PaPSAC}). Here we develop a similar approach for the groups $\Spin{10} \subset \SU{16}$ and $\Spin{10} \cdot \mathrm{U}(1) \subset \mathrm{U}(16)$. Our point of view fits in both contexts of Clifford systems and of even Clifford structures. These two notions are generally closely related, but not equivalent, as we will see in the present situation (cf.\ Theorem \ref{half-spin}).

The notion of \emph{Clifford system}, by definition a family $(P_0, \dots , P_m)$ of symmetric orthogonal and anticommuting endomorphisms in $\RR^N$, has been introduced in the early 1980s by D. Ferus, H. Karcher and H. F. M\"untzer in the framework of isoparametric hypersurfaces of spheres, and has been recently exploited in the study of singular Riemannian foliations in spheres (cf.\ Section \ref{h} for further informations). The second notion, of a \emph{Clifford structure}, has been instead proposed by A. Moroianu and U. Semmelmann, see \cite{MoSCSR,MoPHRH}, and studied in different contexts, being a unifying setting including K\"ahler, quaternion-K\"ahler, $\Spin{7}$, $\Spin{9}$ and other geometries. 

In this paper we describe a rank 10 vector sub-bundle $E \subset \End{TM}$ on a 16-dimensional Hermitian manifold $M$, equipped with a $\Spin{10} \cdot \U{1}$-structure, such that $\Lambda^2 E$ is mapped to the bundle $\Endminus{TM}$ of skew-symmetric endomorphisms. This is exactly the definition of an even Clifford structure (cf.\ Section \ref{Rosenfeld} for more details). As in the case of $\Spin{9}$, one can write (local) skew-symmetric matrices $\psi^D = (\psi_{\alpha \beta})_{0 \leq \alpha, \beta \leq 9}$ of the K\"ahler 2-forms associated with $\Lambda^2 E$.  According to \cite{MoSCSR}, when this structure is parallel with respect to the Levi-Civita connection and on the simply connected non flat case, the Hermitian symmetric space $\EIII$ is the only possibility for such a manifold.

A class of examples of even Clifford structures with $E \subset \End{TM}$ (now $M$ is a real Riemannian manifold of any dimension) is when $E$ is locally spanned by local Clifford systems of involutions on the tangent bundles, and related in the intersections by special orthogonal transformations. We will call \emph{essential} an even Clifford structure which is \emph{not} in this mentioned class of examples.

In this respect, we prove the following:
\begin{te}\label{half-spin} The flat space $\RR^{32}$ admits both a Clifford system $C_9= (\PP_0, \dots , \PP_9)$ and an essential even Clifford structure $E$, related with representations of the abstract group $\Spin{10}$. Namely, the Clifford system $C_9$ defines a real representation of $\Spin{10}$ in $\RR^{32}$, and the even Clifford structure $E$ defines (one of the two conjugate) half-spin representations of $\Spin{10}$ in $\CC^{16}$. 

Moreover, this essential even Clifford structure $E$ is globally defined on the Hermitian symmetric space $\EIII$. Thus here $E \subset \End{T\,\EIII}$, and it is locally described on $\EIII$ by skew-symmetric matrices
\[
\psi^D = (\psi_{\alpha \beta})_{0 \leq \alpha, \beta \leq 9}
\]
of local differential 2-forms, whose fourth coefficient $\tau_4$ of the characteristic polynomial gives a global closed differential 8-form
\[
\Phi_{\Spin{10}} = \tau_4(\psi^D).
\]
\end{te}

Next, we look at $\EIII$ in another aspect, namely as a smooth complex projective algebraic variety. In this respect $\EIII$ has been called the \emph{fourth Severi variety}. This term refers more generally to the possibility of defining projective planes over four complex composition algebras, namely over $\CC \otimes \RR$, $\CC \otimes \CC$, $\CC \otimes \HH$, $\CC \otimes \OO$, and embedded in complex projective spaces of suitable dimension. One can in fact construct in a unified way (cf.\ \cite{LaMPGF}) projective planes over the four listed composition algebras, and get in this way the \emph{four Severi varieties} $V_{(1)},V_{(2)},V_{(3)},V_{(4)}$ as smooth complex projective varieties respectively in $\CC P^5$, $\CC P^8$, $\CC P^{14}$, $\CC P^{26}$. Both the ambient spaces and the Severi varieties can be seen as projectified objects, the former of the Jordan algebra of Hermitian matrices, and the latter of their sets of rank one matrices. Further informations on the Severi varieties will be given in Section \ref{Severi}.

In Section \ref{Rosenfeld} we prove our main result:

\begin{te}\label{main} Let $\omega$ be the K\"ahler form and let $\Phi_{\Spin{10}} = \tau_4(\psi^D)$ be the 8-form on $\EIII$ defined in Theorem\ \ref{half-spin}. Then:

(i) The de Rham cohomology algebra $H^*(\EIII)$ is generated by (the classes of) $\omega \in \Lambda^2$ and $\Phi_{\Spin{10}} \in \Lambda^8$.

(ii) By looking at $\EIII$ as the fourth Severi variety $V_{(4)} \subset \CC P^{26}$, the de Rham dual of the basis represented in $H^8(\EIII; \ZZ)$ by the forms $(\frac{1}{(2\pi)^4} \Phi_{\Spin{10}}, \frac{1}{(2\pi)^4} \omega^4)$ is given by the pair of algebraic cycles
\[
\Big(\CC P^4 + 3(\CC P^4)', \; \;  \CC P^4 + 5(\CC P^4)' \Big),
\]
where $\CC P^4, (\CC P^4)'$ are maximal linear subspaces, belonging to the two different families ruling a totally geodesic non-singular quadric $Q_8$ contained in $V_{(4)}$.
\end{te}

\section{Preliminaries}\label{preliminaries}

A natural approach to $\Spin{10}$-structures is via an extension of the following notion, used in real 16-dimensional Riemannian geometry (see \cite{FriWSS} for $\Spin{9}$-manifolds, and \cite{PaPSAC} for some applications).

\begin{de}\label{defspin9}
A $\Spin{9}$-\emph{structure} on a 16-dimensional Riemannian manifold $(M,g)$ is a rank $9$ real vector bundle
\[
E^9\subset\End{TM}\rightarrow M,
\]
locally spanned by self-dual anti-commuting involutions $\I_{\alpha}: TM \rightarrow TM$. Thus
\begin{equation}\label{1}
 \qquad \qquad \qquad \qquad \qquad \qquad \qquad  \I_{\alpha}^2 = \Id,  \qquad g(\I_{\alpha} X,Y)=g(X,\I_{\alpha}Y), \qquad \qquad (\alpha=1,\dots,9),
\end{equation}
and
\begin{equation}\label{2}
\qquad \qquad \qquad \qquad \qquad \qquad \qquad \qquad \I_{\alpha}\circ\I_{\beta}=-\I_{\beta}\circ\I_{\alpha}, \qquad \qquad \qquad \qquad \qquad (\alpha\neq\beta).
\end{equation}
\end{de}

In the terminology of the Introduction, $E^9$ is a non-essential even Clifford structure, defined through the local Clifford systems $(\I_1, \dots \I_9)$. From these data one gets on $M$ the local almost complex structures
$\J_{\alpha\beta} =\I_{\alpha}\circ\I_{\beta}$,
and the $9\times 9$ skew-symmetric matrix of their K\"ahler $2$-forms
\begin{equation}\label{kae}
\psi^C=(\psi_{\alpha\beta}).
\end{equation} The differential forms $\psi_{\alpha\beta}$, $\alpha<\beta$, are thus \emph{a local system of K\"ahler $2$-forms} on the $\Spin{9}$-manifold $(M^{16},E^9)$.

On the model space $\RR^{16}$, the standard $\Spin{9}$-structure is defined by the generators $\I_1, \dots, \I_9$ of the Clifford algebra $\Cl(9)$, the endomorphisms' algebra of its $16$-dimensional spin real representation $\Delta_9 = \RR^{16} = \OO^2$.
Accordingly, unit vectors in $\RR^9$ can be viewed as self-dual endomorphisms
\[
w: \Delta_9 = \OO^2 \rightarrow \Delta_9 = \OO^2,
\]
and the action of $w = u + r \in S^8$ ($u \in \OO$, $r \in \RR$, $u\overline u + r^2 =1$) on pairs $(x,x') \in \OO^2$ is given by
\begin{equation}\label{HarSpC}
\left(
\begin{array}{c}
x \\
x'
\end{array}
\right)
\longrightarrow
\left(
\begin{array}{cc}
r & R_{\overline u} \\
R_u & -r
\end{array}
\right)
\left(
\begin{array}{c}
x \\
x'
\end{array}
\right),
\end{equation}
where $R_u, R_{\overline u}$ denote the right multiplications by $u, \overline u$, respectively  (cf.~\cite[page 288]{HarSpC}). The choices
\[
(w = 0,1),(0,i),(0,j),(0,k),(0,e),(0,f),(0,g),(0,h)\quad\text{and}\quad(1,0) \in S^8 \subset \OO \times \RR = \RR^9
\]
define the symmetric orthogonal endomorphisms:
\begin{equation}\label{top}
\begin{split}
\I_1=\left(
\begin{array}{c|c}
0 & \Id\\
\hline
\Id & 0
\end{array}
\right), \;
\I_2=\left(
\begin{array}{c|c}
0 & -R_i\\
\hline
R_i & 0
\end{array}
\right), \; \;
&\I_3=\left(
\begin{array}{c|c}
0 & -R_j\\
\hline
R_j & 0
\end{array}
\right), \;
\I_4=\left(
\begin{array}{c|c}
0 & -R_k\\
\hline
R_k & 0
\end{array}
\right),
\\
\I_5=\left(
\begin{array}{c|c}
0 & -R_e\\
\hline
R_e& 0
\end{array}
\right), \;
\I_6=\left(
\begin{array}{c|c}
0 & -R_f\\
\hline
R_f& 0
\end{array}
\right), \;
\I_7=&\left(
\begin{array}{c|c}
0 & -R_g\\
\hline
R_g & 0
\end{array}
\right), \;
\I_8=\left(
\begin{array}{c|c}
0 & -R_h\\
\hline
R_h & 0
\end{array}
\right), \;
\I_9=\left(
\begin{array}{c|c}
\Id & 0\\
\hline
0 & -\Id
\end{array}
\right),
\end{split}
\end{equation}
where $R_i,\dots,R_h$ are the right multiplications by the 7 unit octonions $i,\dots,h$.

The subgroup $\Spin{9} \subset \SO{16}$ is then characterized as preserving the vector space
\begin{equation}\label{eq:V9}
E^9 = <\I_1,\dots,\I_9>\subset \End{\RR^{16}},
\end{equation}
and the space $\Lambda^2\RR^{16}$ of $2$-forms in $\RR^{16}$ decomposes under $\Spin{9}$ as
\begin{equation}\label{decomposition}
\Lambda^2\RR^{16} = \Lambda^2_{36} \oplus \Lambda^2_{84}
\end{equation}
(cf.~\cite[page 146]{FriWSS}), where  $\Lambda^2_{36} \cong \liespin{9}$ and $ \Lambda^2_{84}$ is an orthogonal complement in $\Lambda^2 \cong \lieso{16}$. Bases of the two subspaces are given by the $36$ compositions
\[
\J_{\alpha \beta} = \I_\alpha \I_\beta,
\]
with $\alpha <\beta$, and by the $84$ compositions
\[
\J_{\alpha \beta \gamma} = \I_\alpha \I_\beta \I_\gamma,
\]
$\alpha <\beta<\gamma$, all complex structures on $\RR^{16}$.
We will need the explicit matrices $\J_{\alpha\beta}$.
By using the notation $R_{uv} = R_u \circ R_v$, $u,v \in \OO$, we can arrange their list into the following three families:
\begin{footnotesize}
\begin{equation}\label{eq:J1}
\begin{aligned}
\J_{12}&=\left(
\begin{array}{c|c}
\RO_i & 0 \\ \hline
0 & -\RO_i
\end{array}
\right),\thinspace &
\J_{13}&=\left(
\begin{array}{c|c}
\RO_j & 0 \\ \hline
0 & -\RO_j
\end{array}
\right),\thinspace &
\J_{14}&=\left(
\begin{array}{c|c}
\RO_k & 0 \\ \hline
0 & -\RO_k
\end{array}
\right),\thinspace &
\J_{15}&=\left(
\begin{array}{c|c}
\RO_e & 0 \\ \hline
0 & -\RO_e
\end{array}
\right),
\\
\J_{16}&=\left(
\begin{array}{c|c}
\RO_f & 0 \\ \hline
0 & -\RO_f
\end{array}
\right),\thinspace &
\J_{17}&=\left(
\begin{array}{c|c}
\RO_g & 0 \\ \hline
0 & -\RO_g
\end{array}
\right),\thinspace &
\J_{18}&=\left(
\begin{array}{c|c}
\RO_h & 0 \\ \hline
0 & -\RO_h
\end{array}
\right),\thinspace &
\end{aligned}
\end{equation}
\end{footnotesize}

\begin{footnotesize}
\begin{equation}\label{eq:J2}
\begin{aligned}
\J_{23}&=\left(
\begin{array}{c|c}
-\RO_{ij} & 0 \\ \hline
0 & -\RO_{ij}
\end{array}
\right), \thinspace &
\J_{24}&=\left(
\begin{array}{c|c}
-\RO_{ik} & 0 \\ \hline
0 & -\RO_{ik}
\end{array}
\right),\thinspace &
\J_{25}&=\left(
\begin{array}{c|c}
-\RO_{ie} & 0 \\ \hline
0 & -\RO_{ie}
\end{array}
\right),
\\
\J_{26}&=\left(
\begin{array}{c|c}
-\RO_{if} & 0 \\ \hline
0 & -\RO_{if}
\end{array}
\right), \thinspace &
\J_{27}&=\left(
\begin{array}{c|c}
-\RO_{ig} & 0 \\ \hline
0 & -\RO_{ig}
\end{array}
\right), \thinspace &
\J_{28}&=\left(
\begin{array}{c|c}
-\RO_{ih} & 0 \\ \hline
0 & -\RO_{ih}
\end{array}
\right),
\\
\J_{34}&=\left(
\begin{array}{c|c}
-\RO_{jk} & 0 \\ \hline
0 & -\RO_{jk}
\end{array}
\right),\thinspace &
\J_{35}&=\left(
\begin{array}{c|c}
-\RO_{je} & 0 \\ \hline
0 & -\RO_{je}
\end{array}
\right),\thinspace &
\J_{36}&=\left(
\begin{array}{c|c}
-\RO_{jf} & 0 \\ \hline
0 & -\RO_{jf}
\end{array}
\right),
\\
\J_{37}&=\left(
\begin{array}{c|c}
-\RO_{jg} & 0 \\ \hline
0 & -\RO_{jg}
\end{array}
\right),\thinspace &
\J_{38}&=\left(
\begin{array}{c|c}
-\RO_{jh} & 0 \\ \hline
0 & -\RO_{jh}
\end{array}
\right),\thinspace &
\J_{45}&=\left(
\begin{array}{c|c}
-\RO_{ke} & 0 \\ \hline
0 & -\RO_{ke}
\end{array}
\right),
\\
\J_{46}&=\left(
\begin{array}{c|c}
-\RO_{kf} & 0 \\ \hline
0 & -\RO_{kf}
\end{array}
\right), \thinspace &
\J_{47}&=\left(
\begin{array}{c|c}
-\RO_{kg} & 0 \\ \hline
0 & -\RO_{kg}
\end{array}
\right),\thinspace &
\J_{48}&=\left(
\begin{array}{c|c}
-\RO_{kh} & 0 \\ \hline
0 & -\RO_{kh}
\end{array}
\right),
\\
\J_{56}&=\left(
\begin{array}{c|c}
-\RO_{ef} & 0 \\ \hline
0 & -\RO_{ef}
\end{array}
\right),\thinspace &
\J_{57}&=\left(
\begin{array}{c|c}
-\RO_{eg} & 0 \\ \hline
0 & -\RO_{eg}
\end{array}
\right), \thinspace &
\J_{58}&=\left(
\begin{array}{c|c}
-\RO_{eh} & 0 \\ \hline
0 & -\RO_{eh}
\end{array}
\right),\\
\J_{67}&=\left(
\begin{array}{c|c}
-\RO_{fg} & 0 \\ \hline
0 & -\RO_{fg}
\end{array}
\right),\thinspace &
\J_{68}&=\left(
\begin{array}{c|c}
-\RO_{fh} & 0 \\ \hline
0 & -\RO_{fh}
\end{array}
\right),\thinspace &
\J_{78}&=\left(
\begin{array}{c|c}
-\RO_{gh} & 0 \\ \hline
0 & -\RO_{gh}
\end{array}
\right),
\end{aligned}
\end{equation}
\end{footnotesize}

\begin{footnotesize}
\begin{equation}\label{eq:J3}
\begin{aligned}
\J_{19}&=\left(
\begin{array}{c|c}
0 & -\Id \\ \hline
\Id & 0
\end{array}
\right),\thinspace &
\J_{29}&=\left(
\begin{array}{c|c}
0 & \RO_i \\ \hline
\RO_i & 0
\end{array}
\right),\thinspace &
\J_{39}&=\left(
\begin{array}{c|c}
0 & \RO_j \\ \hline
\RO_j & 0
\end{array}
\right),\thinspace &
\J_{49}&=\left(
\begin{array}{c|c}
0 & \RO_k \\ \hline
\RO_k & 0
\end{array}
\right),\\
\J_{59}&=\left(
\begin{array}{c|c}
0 & \RO_e \\ \hline
\RO_e & 0
\end{array}
\right),\thinspace &
\J_{69}&=\left(
\begin{array}{c|c}
0 & \RO_f \\ \hline
\RO_f & 0
\end{array}
\right),\thinspace &
\J_{79}&=\left(
\begin{array}{c|c}
0 & \RO_g \\ \hline
\RO_g & 0
\end{array}
\right),\thinspace &
\J_{89}&=\left(
\begin{array}{c|c}
0 & \RO_h \\ \hline
\RO_h & 0
\end{array}
\right).
\end{aligned}
\end{equation}
\end{footnotesize}

To write their associated K\"ahler forms $\psi_{\alpha \beta}$, denote the coordinates in $\OO^2\cong\RR^{16}$ just by $(1,\dots,8,1',\dots,8')$. For example, the notation $-\shortform{12}$ stands for the 2-form $ -dx_1 \wedge dx_2$.
\begin{equation}\label{28}
\begin{aligned}
\psi_{12}&=(-\shortform{12}+\shortform{34}+\shortform{56}-\shortform{78})-(\enspace)\shortform{'},\qquad & \psi_{13}&=(-\shortform{13}-\shortform{24}+\shortform{57}+\shortform{68})-(\enspace)\shortform{'},
\\
\psi_{14}&=(-\shortform{14}+\shortform{23}+\shortform{58}-\shortform{67})-(\enspace)\shortform{'},\qquad & \psi_{15}&=(-\shortform{15}-\shortform{26}-\shortform{37}-\shortform{48})-(\enspace)\shortform{'},
\\
\psi_{16}&=(-\shortform{16}+\shortform{25}-\shortform{38}+\shortform{47})-(\enspace)\shortform{'},\qquad & \psi_{17}&=(-\shortform{17}+\shortform{28}+\shortform{35}-\shortform{46})-(\enspace)\shortform{'},
\\
\psi_{18}&=(-\shortform{18}-\shortform{27}+\shortform{36}+\shortform{45})-(\enspace)\shortform{'},\qquad & \psi_{23}&=(-\shortform{14}+\shortform{23}-\shortform{58}+\shortform{67})+(\enspace)\shortform{'},
\\
\psi_{24}&=(\shortform{13}+\shortform{24}+\shortform{57}+\shortform{68})+(\enspace)\shortform{'},\qquad & \psi_{25}&=(-\shortform{16}+\shortform{25}+\shortform{38}-\shortform{47})+(\enspace)\shortform{'},
\\
\psi_{26}&=(\shortform{15}+\shortform{26}-\shortform{37}-\shortform{48})+(\enspace)\shortform{'},\qquad &
\psi_{27}&=(\shortform{18}+\shortform{27}+\shortform{36}+\shortform{45})+(\enspace)\shortform{'},
\\
\psi_{28}&=(-\shortform{17}+\shortform{28}-\shortform{35}+\shortform{46})+(\enspace)\shortform{'},\qquad & \psi_{34}&=(-\shortform{12}+\shortform{34}-\shortform{56}+\shortform{78})+(\enspace)\shortform{'},
\\
\psi_{35}&=(-\shortform{17}-\shortform{28}+\shortform{35}+\shortform{46})+(\enspace)\shortform{'},\qquad & \psi_{36}&=(-\shortform{18}+\shortform{27}+\shortform{36}-\shortform{45})+(\enspace)\shortform{'},
\\
\psi_{37}&=(+\shortform{15}-\shortform{26}+\shortform{37}-\shortform{48})+(\enspace)\shortform{'},\qquad & \psi_{38}&=(\shortform{16}+\shortform{25}+\shortform{38}+\shortform{47})+(\enspace)\shortform{'},
\\
\psi_{45}&=(-\shortform{18}+\shortform{27}-\shortform{36}+\shortform{45})+(\enspace)\shortform{'},\qquad & \psi_{46}&=(\shortform{17}+\shortform{28}+\shortform{35}+\shortform{46})+(\enspace)\shortform{'},
\\
\psi_{47}&=(-\shortform{16}-\shortform{25}+\shortform{38}+\shortform{47})+(\enspace)\shortform{'},\qquad & \psi_{48}&=(\shortform{15}-\shortform{26}-\shortform{37}+\shortform{48})+(\enspace)\shortform{'},
\\
\psi_{56}&=(-\shortform{12}-\shortform{34}+\shortform{56}+\shortform{78})+(\enspace)\shortform{'},\qquad & \psi_{57}&=(-\shortform{13}+\shortform{24}+\shortform{57}-\shortform{68})+(\enspace)\shortform{'},
\\
\psi_{58}&=(-\shortform{14}-\shortform{23}+\shortform{58}+\shortform{67})+(\enspace)\shortform{'},\qquad & \psi_{67}&=(\shortform{14}+\shortform{23}+\shortform{58}+\shortform{67})+(\enspace)\shortform{'},
\\
\psi_{68}&=(-\shortform{13}+\shortform{24}-\shortform{57}+\shortform{68})+(\enspace)\shortform{'},\qquad &
\psi_{78}&=(\shortform{12}+\shortform{34}+\shortform{56}+\shortform{78})+(\enspace)\shortform{'},
\end{aligned}
\end{equation}
where $(\enspace)\shortform{'}$ is for $\shortform{'}$ of what appears before it, for instance
\begin{equation}
\psi_{12}=(-\shortform{12}+\shortform{34}+\shortform{56}-\shortform{78})-(-\shortform{1'2'}+\shortform{3'4'}+\shortform{5'6'}-\shortform{7'8'}).
\end{equation}
Moreover:
\begin{equation}\label{8}
\begin{aligned}
\psi_{19}&=-\shortform{11'}-\shortform{22'}-\shortform{33'}-\shortform{44'}-\shortform{55'}-\shortform{66'}-\shortform{77'}-\shortform{88'},
\\
\psi_{29}&=-\shortform{12'}+\shortform{21'}+\shortform{34'}-\shortform{43'}+\shortform{56'}-\shortform{65'}-\shortform{78'}+\shortform{87'},
\\
\psi_{39}&=-\shortform{13'}-\shortform{24'}+\shortform{31'}+\shortform{42'}+\shortform{57'}+\shortform{68'}-\shortform{75'}-\shortform{86'},
\\
\psi_{49}&=-\shortform{14'}+\shortform{23'}-\shortform{32'}+\shortform{41'}+\shortform{58'}-\shortform{67'}+\shortform{76'}-\shortform{85'},
\\
\psi_{59}&=-\shortform{15'}-\shortform{26'}-\shortform{37'}-\shortform{48'}+\shortform{51'}+\shortform{62'}+\shortform{73'}+\shortform{84'},
\\
\psi_{69}&=-\shortform{16'}+\shortform{25'}-\shortform{38'}+\shortform{47'}-\shortform{52'}+\shortform{61'}-\shortform{74'}+\shortform{83'},
\\
\psi_{79}&=-\shortform{17'}+\shortform{28'}+\shortform{35'}-\shortform{46'}-\shortform{53'}+\shortform{64'}+\shortform{71'}-\shortform{82'},
\\
\psi_{89}&=-\shortform{18'}-\shortform{27'}+\shortform{36'}+\shortform{45'}-\shortform{54'}-\shortform{63'}+\shortform{72'}+\shortform{81'}.
\end{aligned}
\end{equation}
All of this, via invariant polynomials, allows to get global differential forms on manifolds $M^{16}$.

The following is proved in \cite{PaPSAC}:

\begin{pr}\label{acs:7->8->9}
%
%
Let
\[
\psi^C =(\psi_{\alpha \beta})_{1 \leq \alpha, \beta \leq 9}
\]
be the skew-symmetric matrix of the K\"ahler 2-forms associated with the family of complex structures $J_{\alpha \beta}$. If $\tau_2$ and $\tau_4$ denote the second and fourth coefficient of the characteristic polynomial, then:
\[
\tau_2(\psi^C)=0, \qquad  \frac{1}{360} \tau_4(\psi^C) = \Phi_{\Spin{9}},
\]
where
$ \Phi_{\Spin{9}} \in \Lambda^8(\RR^{16})$ is the canonical form associated with
the standard $\Spin{9}$-structure in $\RR^{16}$.
%
%
%
\end{pr}

The 8-form $\Phi_{\Spin{9}}$ was originally defined by M. Berger in 1972, cf.\ \cite{BerCCP}. See also the following Remark \ref{berger}.


\bigskip

\section{The fourth Severi variety $V_{(4)} \subset \CC P^{26}$}\label{Severi}

The following characterization of the four Severi varieties was proved by F. L. Zak in the early 1980s in the context of chordal varieties (\cite{ZakSeV,LaVTGP}). Let $V_n \subset \CC P^m$ be a smooth complex projective variety not contained in a hyperplane, and assume that its dimension $n$ satisfies $n=\frac{2}{3}(m-2)$. Then the chordal variety $\mathrm{Chord}\;  V$, locus of of all the secant and tangent lines, coincides with $\CC P^m$, unless $n=2,4,8,16$ and $V$ is one of the following projective varieties:

i) $V_{(1)} \cong \CC P^2$, the Veronese surface in $\CC P^5$,

ii) $V_{(2)} \cong \CC P^2 \times \CC P^2$, the Segre four-fold in $\CC P^8$,

iii) $V_{(3)} \cong Gr_2(\CC^6)$, the Pl\"ucker embedding of this Grassmannian in $\CC P^{14}$,

iv) $V_{(4)} \cong \EIII$, the projective plane over complex octonions as a smooth subvariety of $\CC P^{26}$.

Moreover, the lower codimension hypothesis $n> \frac{2}{3} (m-2)$ insures that $\mathrm{Chord}\;  V = \CC P^m$.

For the four mentioned exceptions, namely the Severi varieties $V_{(i)}$ $(i=1,2,3,4)$, the chordal variety $\mathrm{Chord}\;  V_{(i)}$ coincides with the cubic hypersurface $\det A =0$, i.e.\ with the variety of matrices of rank $\leq 2$ in the construction via the Jordan algebra $\mathcal H_3$ in the respective complex composition algebra.

The name for these four $V_{(i)}$ was given by Zak in recognition of a 1901 F. Severi's work \cite{SevIPD}, investigating projective surfaces with the mentioned chordal property, and characterizing in this way the Veronese surface $V_{(1)}$ of $\CC P^5$. It is notable that from Zak classification it follows that all the four Severi varieties can be looked at ``Veronese surfaces", i.e.\ at projective planes
\begin{equation}
\begin{split}
V_{(1)}=(\CC \otimes \RR) P^2 \subset \CC P^5, &\quad V_{(2)}=(\CC \otimes \CC)P^2 \subset \CC P^8, \\
V_{(3)}=(\CC \otimes \HH) P^2 \subset \CC P^{14}, &\quad V_{(4)}=(\CC \otimes \OO) P^2 \subset \CC P^{26}
\end{split}
\end{equation}
embedded in complex projective spaces via an appropriately written ``Veronese map"
\[
\qquad \qquad \qquad \qquad \qquad \qquad \qquad \qquad \qquad (x_0:x_1:x_2) \longrightarrow (\dots : x_l \bar x_m: \dots)  \qquad \qquad \qquad \qquad (0 \leq l\leq m \leq 2),
\]
where $\bar x_m$ denotes the conjugation in the second factor algebra (cf.\ \cite[Theorems 6 and 7]{ZakSeV}).

It is relevant for us that the four Severi varieties appear in the following table of ``projective planes" $(\mathbb K \otimes \mathbb K')P^2$:

\bigskip
\[
\renewcommand{\arraystretch}{1.65}
\resizebox*{1.00\textwidth}{!}{
\begin{tabular}{|c||c|c|c|c|}
\hline
$_{\mathbb K \; =} \backslash ^{\mathbb K' \; =}$  & $\mathbb R$ &$\mathbb C$&$\mathbb H$&$\mathbb O$\\
\hline\hline
{$\RR$} &{$\RR P^2$}& $\CC P^2 \cong V_2^4$ & {
$\HH P^2 $}& {$\OO P^2 \cong \Ffour/\Spin{9}$}\\
\hline
{$\CC$}&$\CC P^2 \cong V_2^4$ &$\CC P^2 \times \CC P^2 \cong V_4^6$ &$Gr_2(\CC^6) \cong V_8^{14}$ & $\Esix/ \mathrm{Spin}(10) \cdot \mathrm{U}(1)  \cong V_{16}^{78}$\\
\hline
{$\HH$} &{$\HH P^2$} & $Gr_2(\CC^6) \cong V_8^{14}$& {$Gr_4^{or}(\RR^{12})$} & {$\mathrm{E_7}/\mathrm{Spin}(12) \cdot \Sp{1} $}\\
\hline
{$\OO$} &{$\OO P^2 \cong \Ffour/ \Spin{9}$} & $\Esix/ \mathrm{Spin}(10) \cdot \mathrm{U}(1) \cong V_{16}^{78}$&{$\mathrm{E_7}/ \mathrm{Spin}(12) \cdot \Sp{1}$}&{$\mathrm{E_8}/ \mathrm{Spin}(16)^+$} \\
\hline
\end{tabular}
}
\]
\bigskip

This can be seen as an application to compact symmetric spaces of the Freudenthal magic square of Lie algebras, see e.g.\ \cite[page 193]{BaeOct}. In particular, in the $\CC$-row and the $\CC$-column of the above table we see the four Severi varieties
\[
V_{(1)}=V_2^4 \subset \CC P^5, \quad V_{(2)}=V_4^6 \subset \CC P^8, \quad V_{(3)}=V_8^{14} \subset \CC P^{14}, \quad V_{(4)}=V_{16}^{78} \subset \CC P^{26},
\]
where following the classical notations $V_n^d$ denotes a complex projective algebraic variety of dimension $n$ and degree $d$ in a $\CC P^m$.

We can also recognize how the cohomology of the first, second and third Severi variety is generated by the cohomology classes of canonical differential forms. There is of course the K\"ahler 2-form as the only generator on $V_{(1)} \cong \CC P^2$, and there are the two K\"ahler 2-forms of the factors for $V_{(2)} \cong \CC P^2 \times \CC P^2$. On $V_{(3)} \cong Gr_2(\CC^6) $ the cohomology generators are the complex K\"ahler 2-form $\omega$ and the quaternionic 4-form $\Omega$, since $Gr_2(\CC^6) $ turns out to have both a complex K\"ahler and a quaternion-K"ahler structure, with no compatibility between them. Thus one expects something similar for the fourth Severi variety $V_{(4)} \cong \EIII$, where its complex K\"ahler structure may be non-compatible with its ``octonionic" one.

Both the cohomology algebra and the Chow ring of $V_{(4)} \cong \EIII$ have been computed (see \cite{ToWICR,IlMCRC,DuZCRG}). The integral cohomology algebra has no torsion and:
\begin{equation}\label{integral}
H^*(\EIII) \cong \ZZ [a_1,a_4]/(r_{9},r_{12})
\end{equation}
where $a_1  \in H^{2}$, $a_4 \in H^8$ and $r_{\beta}$ denote relations in $H^{2\beta}$.

A CW-decomposition of $\EIII$ into Schubert cycles is described in \cite{IlMCRC,DuZCRG}. This allows to get the Chow ring of $\EIII$, whose structure is isomorphic to that of mentioned cohomology. This has been done in \cite{IlMCRC} by obtaining three generators and several relations, and in \cite{DuZCRG} has been observed that two generators suffice. The following picture, taken from \cite{IlMCRC}, describes the Schubert cycles of $\EIII$, labelled by their degrees, and their incidence relations. The complex dimension of the cycles goes from zero on the left to 16 on the right, where the full fourth Severi variety $V_{16}^{78}$ appears.
\bigskip

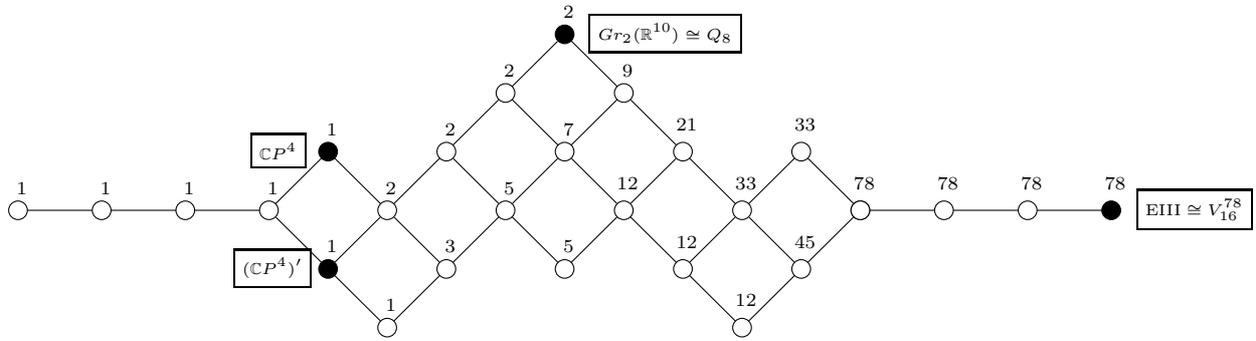
\begin{center}
\begin{tikzpicture}
    \tikzstyle{every node}=[draw,circle,fill=white,minimum size=7pt,
                            inner sep=0pt]

  1.1  
    \draw (0,0) node (82) [fill=black, label=above:$\  _2$, label=right: {\tiny $\; \; ^{\boxed{Gr_2(\RR^{10}) \cong Q_8 }}$}] {}
        -- ++(315:1.1cm) node (99) [label=above:$\ _9$] {}
        -- ++(225:1.1cm) node (87) [label=above:$\ _7$] {}
        -- ++(135:1.1cm) node (72) [label=above:$\ _2$] {}
        --(82); 

    \draw (72)
        -- ++(225:1.1cm) node (62) [label=above:$\ _2$] {}
              -- ++(225:1.1cm) node (52) [label=above:$\ _2$] {}
                    -- ++(135:1.1cm) node (41) [fill=black, label=above:$\ _1$,  label=left: {\tiny $^{\boxed{ \CC P^4}} \;$}]  {}
                          -- ++(225:1.1cm) node (31) [label=above:$\ _1$] {}
                                -- ++(180:1.1cm) node (21) [label=above:$\ _1$] {}
                                             -- ++(180:1.1cm) node (11) [label=above:$\ _1$] {}
                                                          -- ++(180:1.1cm) node (01) [label=above:$\ _1$] {};

    \draw (31)
                 -- ++(315:1.1cm) node (4'1) [fill=black, label=above:$\ _1$,  label=left: {\tiny $^{\boxed{(\CC P^4) '}}$}]  {}
                              -- ++(315:1.1cm) node (5'1) [label=above:$\ _1$] {}
                                           -- ++(45:1.1cm) node (63) [label=above:$\ _3$] {}
                                                      -- ++(45:1.1cm) node (75) [label=above:$\ _5$] {}
                                                                          -- ++(315:1.1cm) node (85) [label=above:$\ _5$] {}
                                                                                    -- ++(45:1.1cm) node (912) [label=above:$\ _{12}$] {}
                                                                                              -- ++(315:1.1cm) node (1012) [label=above:$\ _{12}$] {}
                                                                                                 -- ++(315:1.1cm) node (1112) [label=above:$\ _{12}$] {}
                                                                                                 -- ++(45:1.1cm) node (1245) [label=above:$\ _{45}$] {}
                                                                                                 -- ++(45:1.1cm) node (1378) [label=above:$\ _{78}$] {}
                                                                                                 -- ++(0:1.1cm) node (1478) [label=above:$\ _{78}$] {}
                                                                                                  -- ++(0:1.1cm) node (1578) [label=above:$\ _{78}$] {}
                                                                                                   -- ++(0:1.1cm) node (1678) [fill=black, label=above:$\ _{78}$, label=right: {\tiny $\; \; ^{\boxed{\EIII \cong V^{78}_{16}}}$}]{} ;

     \draw (99)
                 -- ++(315:1.1cm) node (1021) [label=above:$\ _{21}$] {}
                              -- ++(315:1.1cm) node (1133) [label=above:$\ _{33}$] {}
                                           -- ++(45:1.1cm) node (1233) [label=above:$\ _{33}$] {}
                                                      -- ++(315:1.1cm) node (1378) [label=below:$$] {} ;

;
   \draw (4'1) -- (52);
   \draw (52) -- (63);
    \draw (62) -- (75);
    \draw (75) -- (87);
    \draw (87) -- (912);
    \draw (912) -- (1021);
    \draw (1012) -- (1133);
     \draw (1133) -- (1245);

\end{tikzpicture}
%

\end{center}
\vspace{1cm}

The four black nodes appearing in the the diagram emphasize, besides the whole variety $\EIII \cong (\CC \otimes \OO)P^2 \cong V^{78}_{16}$, its totally geodesic ``projective line" $(\CC \otimes \OO)P^1 \cong Gr_2(\RR^{10})$, isometric to a non singular quadric $Q_8 \subset \CC P^9$. It is well known from projective geometry (see for example \cite[page 64]{SegPGA}), that even dimensional non singular quadrics admit two families of maximal linear subspaces. In the case of $Q_8$ these are two 10-dimensional families of 4-dimensional linear subspaces. Elements $\CC P^4$, $(\CC P^4)'$ of these two families appear in the diagram, where it appears that the $(\CC P^4)'$ (but not the $\CC P^4$) are extendable to 5-dimensional linear subspaces in $V^{78}_{16}$, but as mentioned non-extendable in $Q_8$.

\begin{re}\label{subgrassmannians}
\rm{The two families (both of complex dimension 10) of complex projective spaces $\CC P^4 \subset Q_8$ can be viewed also as the families of sub-Grassmannians $Gr_1(\CC^5) \subset Gr_2(\RR^{10})$ with respect to a choice of complex structures on the two families parametrized by the Hermitian symmetric space $\SO{10}/\mathrm{U}(5)$, with respect to the two possible orientations. This observation will be used in the proof of Theorem \ref{main}.}
\end{re}

\section{A Clifford system and a Lie subalgebra $\mathfrak{h} \subset \mathfrak{so}(32)$}\label{h}
%

The construction outlined in Section \ref{preliminaries} for the canonical  8-form $\Phi_{\Spin{9}}$ can be seen in parallel with those of other canonical differential forms.

In particular, the datum of a rank 5 vector bundle $E^5 \subset \End{TM}$ over a Riemannian manifold $M^8$, locally generated by involutions $\I_1, \dots , \I_5$ satisfying properties \eqref{1}, \eqref{2}, is equivalent to the datum of an almost quaternion-Hermitian structure of $M^8$. One sees in particular that the quaternionic 4-form in real dimension 8 can be constructed from $E^5$, cf.\ \cite[page 329]{PaPSAC}. On the other hand, the vector bundles $E^5 \subset \End{TM}$ and $E^9 \subset \End{TM}$, when $M$ is respectively a Riemannian $M^8$ or $M^{16}$, are examples of even Clifford structure in the sense of \cite{MoSCSR}, and they are both non-essential, according to the definition given in the Introduction. Thus, in the mentioned examples, such an even Clifford structure is equivalent to the datum on $M^8$ or $M^{16}$ of a $\Sp{2}\cdot \Sp{1}$ or a $\Spin{9}$-structure.

It is therefore natural to inquire about the possibility of a similar approach for a $\Spin{10}$-structure on $\CC^{16}$, and again more generally on manifolds. The following Proposition shows that the same approach cannot be pursued without modification for the group $\Spin{10}$.

\begin{pr}\label{pr:ind}
The complex space $\CC^{16}$ does not admit any family of ten endomorphisms $\I_0,\dots,\I_9$, satisfying the properties \eqref{1} and \eqref{2} with respect to the standard Hermitian scalar product $g$.
\end{pr}

\begin{proof} Assume that $\I_0,\dots,\I_9$ are involutions on $\CC^{16}$ satisfying \eqref{1} and \eqref{2}. Note that any such $\I_\alpha$ lies necessarily in $\mathrm{U}(16)$. To show the non-existence on $\CC^{16}$ of such a datum, we will see that the set of all compositions
\begin{equation}\label{375}
\begin{split}
\J_{\alpha\beta} = \I_\alpha\I_\beta  \; (\alpha<\beta), \quad \J_{\alpha\beta\gamma} = \I_\alpha\I_\beta\I_\gamma  \; (\alpha<\beta<\gamma), \\
\J_{\alpha\beta\gamma\delta\epsilon\zeta} = \I_\alpha\I_\beta\I_\gamma \I_\delta \I_\epsilon \I_\zeta \; (\alpha<\beta<\gamma<\delta<\epsilon<\zeta) \; \;
\end{split}
\end{equation}
of two, three and six such involutions would give rise to linearly independent complex structures on $\CC^{16}$. In fact, counting their number, we would obtain in this way 45+120+210=375 linearly independent complex structures. But the dimension of the space $\Lambda^{1,1}$ of orthogonal complex structures on $\CC^{16}$ is only $256$.

Going into some details, it is an easy consequence of relations \eqref{1} and \eqref{2} that any composition of 2, 3 or 6 different involutions among $\I_0,\dots,\I_9$ is a complex structure, whereas the composition of 4 or 5 of them is an involution. With this in mind, we can show that the complex structures listed in \eqref{375} are mutually orthogonal. To see this, and as a first observation, we get immediately that $\tr{\J^*_{\alpha\beta}\J_{\gamma\delta}}=\tr{\I_\beta\I_\alpha\I_\gamma\I_\delta}=0$ if any of $\gamma < \delta$ equals any of $\alpha <\beta$. But also $\tr{\J^*_{\alpha\beta}\J_{\gamma\delta}}=\tr{\I_\beta\I_\alpha\I_\gamma\I_\delta}=0$ if  $\alpha\neq\gamma , \delta$ and $\beta\neq \gamma , \delta$, since we are here composing the skew symmetric endomorphism $\J_{\beta\alpha\gamma}$ with the symmetric $\I_\delta$. Thus, any pair in the first family $\{\J_{\alpha\beta}\}$ with $\alpha <\beta$ listed in \eqref{375} is given by orthogonal complex structures. Similarly, one sees that any pair chosen inside the family $\{\J_{\alpha\beta\gamma}\}$, for $\alpha<\beta<\gamma$, or inside the family $\{\J_{\alpha\beta\gamma\delta\epsilon\zeta}\} $, for $\alpha<\beta<\gamma<\delta<\epsilon<\zeta$, is given also by orthogonal complex structures.

It remains to be shown that complex structures chosen in different families among the three listed in \eqref{375} are mutually orthogonal. We do this for $\J_{\alpha\beta} $ and $\J_{\gamma\delta\epsilon}$ and without coincidences between any of the first two and any of the last three indices (when there is at least one coincidence, the orthogonality is immediate). Remind that the composition $\J_{\alpha\beta\gamma\delta\epsilon} = \J_{\alpha\beta} \circ \J_{\gamma\delta\epsilon}$ is a self adjoint involution and, if the index $\zeta$ is distinct from all the five previous ones, $\J_{\alpha\beta\gamma\delta\epsilon}$ anti-commutes with $\I_\zeta$. Thus, we are dealing with the two self adjoint anti-commuting involutions $A=\J_{\alpha\beta \gamma\delta\epsilon} $ and $B=\I_\zeta$. Let $B'=C^{-1}BC = $ diag $=(\lambda_1, \dots \lambda_{16})$ be a diagonal form of $B$ and let $A'=C^{-1}AC$. Since $B'^2=B^2$ is the identity matrix, we have $\lambda_1^2 = \dots = \lambda_{16}^2=1$. On the other hand, from $AB=-BA$ we get $A'B'=-B'A'$, so that the diagonal entries of $A'$ are all zero. It follows that the trace of $A=\J_{\alpha\beta \gamma\delta\epsilon}$ is zero, and thus $\tr{\J^*_{\alpha\beta} \J_{\gamma\delta \epsilon}}=0$, showing that $\J_{\alpha\beta} $ and $\J_{\gamma\delta\epsilon}$ are orthogonal. The other possibilities of two complex structures in different families listed in \eqref{375} are treated in a similar way.
\end{proof}

Both definitions of a $\Sp{2} \cdot \Sp{1}$-structure and of a $\Spin{9}$-structure (cf.\ Definition \ref{defspin9} and the discussion at the very beginning of this Section) fit in the framework of the so-called \emph{Clifford systems} (see \cite{FKMCNI,RadCAN,GoRHCC}).
These are sets $C=(P_0, \dots,P_m)$ of symmetric transformations in a Euclidean real vector space $\RR^N$ such that $P_\alpha^2 = \Id$ for all $\alpha$ and $P_\alpha P_\beta+P_\beta P_\alpha=0$ for all $\alpha \neq \beta$. One can then show that a Clifford system exists in a $\RR^N$ if and only if $N=2k\delta(m)$, where $k$ is a positive integer and $\delta(m)$ is given by:

\medskip
\renewcommand{\arraystretch}{1.85}
\begin{center}
\medskip
\begin{tabular}{|p{.65in}  ||c|c|c|c|c|c|c|c|c|c|}
\hline
$\quad \; \; m$&$1$&$2$&$3$&$4$&$5$&$6$&$7$&$8$&$8+h$\\
\hline
$\; \; \; \; \delta(m)$&$1$&$2$&$4$&$4$&$8$&$8$&$8$&$8$&$16\delta(h)$\\
\hline
\end{tabular}
\end{center}
\bigskip

\noindent When $k=1$ the Clifford system is said to be \emph{irreducible}. Thus, for $m=8$ and for $m=4$, an irreducible Clifford system defines a $\Spin{9}$ and a $\Sp{2} \cdot \Sp{1}$ structure in $\RR^{16}$ and in $\RR^8$, respectively. The prototype example of an irreducible Clifford system is, for $m=2$ and in $\RR^4 \equiv \CC^2$, the set of the three Pauli matrices
\[
P_0=\left(
\begin{array}{rr}
0 & 1\\
1 & 0
\end{array} \right), \;\;P_1= \left(
\begin{array}{rr}
0& i\\
-i&0
\end{array}\right),\; \; P_2 = \left(
\begin{array}{rr}
1& 0\\
0 & -1
\end{array}\right) \; \in \; \mathrm{U}(2).
\]

\noindent The former table foresees the existence of an irreducible Clifford system with $m=9$ in the vector space $\RR^{32}$. Be careful that this does not contradict Proposition \ref{pr:ind}, stating that such a Clifford system cannot be chosen with all elements in $\U{16}$. To write a Clifford system $C_9=(\PP_0, \PP_1, \dots , \PP_9)$ in $\RR^{32}$, one can imitate the procedure that allows to pass from $C_4$ to $C_8$, i.e.\ from a $\Sp{2}\cdot \Sp{1}$ to a $\Spin{9}$ structure. This gives the following symmetric matrices in $\SO{32}$:
\begin{equation}\label{top32}
\begin{split}
\PP_0=\left(
\begin{array}{c|c}
0 & \Id\\
\hline
\Id & 0
\end{array}
\right), \;
\PP_1=\left(
\begin{array}{c|c}
0 & -J_{12}\\
\hline
J_{12}& 0
\end{array}
\right)&, \;
\PP_2=\left(
\begin{array}{c|c}
0 & -J_{13}\\
\hline
J_{13} & 0
\end{array}
\right), \;
\PP_3=\left(
\begin{array}{c|c}
0 & -J_{14}\\
\hline
J_{14}& 0
\end{array}
\right),
\\
\PP_4=\left(
\begin{array}{c|c}
0 & -J_{15}\\
\hline
J_{15} & 0
\end{array}
\right), \;
\PP_5=&\left(
\begin{array}{c|c}
0 & -J_{16}\\
\hline
J_{16}& 0
\end{array}
\right), \;
\PP_6=\left(
\begin{array}{c|c}
0 & -J_{17}\\
\hline
J_{17}& 0
\end{array}
\right),
\\
\PP_7=\left(
\begin{array}{c|c}
0 & -J_{18}\\
\hline
J_{18}& 0
\end{array}
\right), \;
\PP_8=&\left(
\begin{array}{c|c}
0 & -J_{19}\\
\hline
J_{19} & 0
\end{array}
\right), \;
\PP_9=\left(
\begin{array}{c|c}
\Id & 0\\
\hline
0 & -\Id
\end{array}
\right),
\end{split}
\end{equation}
where the $J_{1\beta}$ are the ones defined in \eqref{eq:J1}, \eqref{eq:J3}. It is immediately checked that $\PP_\alpha^2=\Id$ and $\PP_\alpha \PP_\beta = - \PP_\beta \PP_\alpha$ for $\alpha \neq \beta$.

The following complex structures $P_{\alpha \beta} = \PP_\alpha \circ \PP_\beta$ ($\alpha <\beta$) in $\RR^{32}$ generate a Lie subalgebra $\mathfrak{h} \subset \mathfrak{so}(32)$. We will see in a moment that $\mathfrak{h} \cong \liespin{10} \subset \liesu{16} \subset \mathfrak{so}(32)$. The 45 $P_{\alpha \beta}$ can be split into the following three families of respectively $8, 28$ and $9$ skew-symmetric matrices:
\begin{footnotesize}
\begin{equation}\label{eq:P1}
P_{01}=\left(
\begin{array}{c|c}
\J_{12}& 0 \\ \hline
0 & -\J_{12}
\end{array}
\right),\thinspace
P_{02}=\left(
\begin{array}{c|c}
\J_{13}& 0 \\ \hline
0 & -\J_{13}
\end{array}
\right),\thinspace
\dots \dots ,\thinspace
P_{08}=\left(
\begin{array}{c|c}
\J_{19}& 0 \\ \hline
0 & -\J_{19}
\end{array}
\right),
\end{equation}
\end{footnotesize}

\begin{footnotesize}
\begin{equation}\label{eq:P2}
P_{12}=\left(
\begin{array}{c|c}
\J_{23} & 0 \\ \hline
0 & \J_{23}
\end{array}
\right), \thinspace
P_{13}=\left(
\begin{array}{c|c}
\J_{24}& 0 \\ \hline
0 & \J_{24}
\end{array}
\right),\thinspace
\dots \dots , \thinspace
P_{78}=\left(
\begin{array}{c|c}
\J_{89}& 0 \\ \hline
0 & \J_{89}
\end{array}
\right),
\end{equation}
\end{footnotesize}

\begin{footnotesize}
\begin{equation}\label{eq:P3}
P_{09}=\left(
\begin{array}{c|c}
0 & -\Id \\ \hline
\Id & 0
\end{array}
\right),\thinspace
P_{19}=\left(
\begin{array}{c|c}
0 & \J_{12} \\ \hline
 \J_{12}& 0
\end{array}
\right),\thinspace
P_{29}=\left(
\begin{array}{c|c}
0 &  \J_{13}\\ \hline
 \J_{13}& 0
\end{array}
\right),\thinspace
\dots \dots , \thinspace
P_{89}=\left(
\begin{array}{c|c}
0 & \J_{19}\\ \hline
 \J_{19}& 0
\end{array}
\right).
\end{equation}
\end{footnotesize}

Note that by construction the vector space generated by all these $P_{\alpha\beta}$ is a Lie subalgebra $\mathfrak{h}$ of $\lieso{32}$.


%
\section{The Lie algebra $\mathfrak{spin}(10) \subset \mathfrak{su}(16)$}\label{liespin10}

To relate the Lie algebra
\[
\mathfrak{h}= <P_{\alpha \beta}>_{0\leq \alpha < \beta \leq 9}
\]
constructed in the previous Section with the Lie algebra $\liespin{10}$, it is useful to compare it with a description of the (half-)spin representation of the group $\Spin{10}$ on $\CC^{16}$.  References for spin representations are for example \cite[Lecture 13]{PosLGL} and \cite[Chapter 3]{MeiCAL}. However, a specific (and for us convenient) excellent account to the group $\Spin{10}$ has been given by R. Bryant in the file \cite{BryRSL}. Since the representation of $\Spin{9} \subset \SO{16}$ in Bryant's notes is slightly different from the one used in R. Harvey's book \cite{HarSpC}, and since we used this latter both in our previous papers \cite{PaPSAC,PaPSWM,OPPSGO} and in the previous Sections, we need first to rephrase in our context some arguments.

At Lie algebras level, we can go from $\mathfrak{spin}(9)$ to $\mathfrak{spin}(10)$ by adding to the family $J^C = \{J_{\alpha \beta}\}_{1 \leq \alpha < \beta \leq 9}$ of 36 complex structures nine further complex structures in $\CC^{16}$. Since the spin representation of $\Spin{10}$ is on $\CC^{16}$, the new family
\[
J^D = \{J_{\alpha \beta}\}_{0 \leq \alpha < \beta \leq 9}
\]
will be of 45 complex structures on $\CC^{16}$, and a basis of $\liespin{10}$.

In the approach of \cite{BryRSL}, one first looks at $\Spin{10}$ as a subgroup of $\mathrm{Cl} (\RR \oplus \OO, <,>)$, the Clifford algebra generated by $\RR \oplus \OO$ endowed with its direct sum inner product. This algebra is isomorphic to $\Endcomplex{\CC \otimes \OO^2}$: since this latter is isomorphic to $M_{16}(\CC)$, the linear map
\[
m_{(r,v)} : \CC \otimes \OO^2 \rightarrow \CC \otimes \OO^2
\]
defined by the matrix
\[
m_{(r,v)} = i \left(
\begin{array}{cc}
r & R_{\overline v} \\
R_v & -r
\end{array}
\right)
\]
has to be, by dimensional reasons, a one-to-one onto representation, whence the claimed isomorphism
\[
\mathrm{Cl} (\RR \oplus \OO, <,>) \cong \Endcomplex{\CC \otimes \OO^2}.
\]

This allows to recognize the Lie algebra $\liespin{10}$ as:
\[
\liespin{10} = \Bigg\{ \left(
\begin{array}{rr}
a_+  + i r \Id_8 & R_{\overline u} + i R_{\overline v} \\
- R_u + i R_v & a_-  - i r \Id_8
\end{array}
\right) , \;  r\in \RR, \; u,v \in \OO,\;\ a=(a_+,a_-) \in \liespin{8} \Bigg\} .
\]

This description is consistent with obtaining $\liespin{10}$ through the datum of the nine extra complex structures $J_{01} =\I_0 \I_1, J_{02}=\I_0\I_2, \dots , J_{09}=\I_0\I_9$ to be added to the family $J^C$ of the 36 complex structures defining its Lie sub-algebra
\[
\liespin{9} = \Bigg\{ \left(
\begin{array}{rr}
a_+ & R_{\overline u}  \\
- R_u  & a_-
\end{array}
\right) , \; u \in \OO,\;\ a=(a_+,a_-) \in \liespin{8} \Bigg\} .
\]
In particular, the inclusions
\[
\liespin{9} \subset \lieso{16}, \qquad \liespin{10} \subset \liesu{16}
\]
are immediately recognized.

Observe also that, since there are no intermediate subgroups between $\Spin{9}$  and $\Spin{10}$,  the latter is generated by its subgroup $\Spin{9}$ and by the circle
\[
T= \Bigg\{ \left(
\begin{array}{cc}
e^{ir} \Id_8 &0 \\
0 & e^{  - i r } \Id_8
\end{array}
\right) , \;  r\in \RR,  \Bigg\}.
\]

Moreover, looking back to the the family $J^C$ of complex structures and their K\"ahler forms given by \eqref{28}, \eqref{8}, note that once the 8 complex structures \eqref{8} are given, one can construct from them all the remaining 28 complex structures as:
\begin{equation}
\begin{split}
J_{78} = \frac{1}{2} [J_{89},J_{79}], &
\\
J_{67} = \frac{1}{2} [J_{79},J_{69}], \; J_{68} = \frac{1}{2} [J_{89},J_{69}], &
\\
J_{56} = \frac{1}{2} [J_{69},J_{59}], \; J_{57} = \frac{1}{2} [J_{79},J_{59}], \; J_{58} = \frac{1}{2} [J_{89},J_{59}], &
\\
J_{45} = \frac{1}{2} [J_{59},J_{49}], \; J_{46} = \frac{1}{2} [J_{69},J_{49}], \; J_{47} = \frac{1}{2} [J_{79},J_{49}], \; J_{48} = \frac{1}{2} [J_{89},J_{49}], & \dots
\end{split}
\end{equation}
and so on, up to $J_{12} = \frac{1}{2} [J_{29},J_{19}], \dots, J_{18} = \frac{1}{2} [J_{89},J_{19}]$.

Thus, a coherent way to define new complex structures $J_{01}, J_{02}, \dots , J_{09}$ in $\CC^{16}$, and to obtain thus a basis of $\liespin{10}$, is given as follows:
\begin{equation}\label{9brackets}
\begin{split}
J_{09} = i \left(
\begin{array}{c|c}
\Id & 0\\
\hline
0 & -\Id
\end{array}
\right) &=i \I_9 ,
\\
J_{01} = \frac{1}{2} [J_{19},J_{09}]= i \left(
\begin{array}{c|c}
0 & \Id\\
\hline
\Id& 0
\end{array}
\right)=
i \I_1 ,  &
\qquad J_{02} = \frac{1}{2} [J_{29},J_{09}]= i \left(
\begin{array}{c|c}
0 & -R_i\\
\hline
R_i & 0
\end{array}
\right)
= i \I_2 ,
\\
J_{03} = \frac{1}{2} [J_{39},J_{09}]= i \left(
\begin{array}{c|c}
0 & -R_j\\
\hline
R_j & 0
\end{array}
\right)
= i \I_3 , &
\qquad J_{04} = \frac{1}{2} [J_{49},J_{09}]= i \left(
\begin{array}{c|c}
0 & -R_k\\
\hline
R_k & 0
\end{array}
\right)
= i \I_4 ,
\\
J_{05} = \frac{1}{2} [J_{59},J_{09}]= i \left(
\begin{array}{c|c}
0 & -R_e\\
\hline
R_e & 0
\end{array}
\right)
= i \I_5 , &
\qquad J_{06} = \frac{1}{2} [J_{69},J_{09}]= i \left(
\begin{array}{c|c}
0 & -R_f\\
\hline
R_f & 0
\end{array}
\right)
= i \I_6 ,
\\
J_{07} = \frac{1}{2} [J_{79},J_{09}]= i \left(
\begin{array}{c|c}
0 & -R_g\\
\hline
R_g & 0
\end{array}
\right)
= i \I_7 , &
\qquad J_{08}= \frac{1}{2} [J_{89},J_{09}]= i \left(
\begin{array}{c|c}
0 & -R_h\\
\hline
R_h & 0
\end{array}
\right)
=i \I_8 ,
\end{split}
\end{equation}
where in the definition of $J_{09}$ we took into account the first of the two previous observations and in the definition of $J_{01}, \dots , J_{08}$ the second one.


Denote now by
\[
J^D = \{J_{\alpha \beta}\}_{0 \leq \alpha < \beta \leq 9}
\]
this family of 45 complex structures, a basis of $\liespin{10}$.

The K\"ahler 2-forms $\psi_{\alpha \beta}$ associated with the complex structures in $J^D $ can now be written. The 36 K\"ahler forms in the subfamily $J^C$ are of course deduced from those in \eqref{28}, \eqref{8}.

To complete them, use now notations $(z_1,\dots,z_8,z_{1'},\dots,z_{8'})$=$(1,\dots,8,1',\dots,8')$ for the coordinates in $\CC^{16}$, and the short forms
\[
\shortform{\alpha \bar \beta}, \qquad \shortform{\alpha\bar \beta\gamma \bar\delta}
\]
(smaller size and boldface) to denote
\[
dz_\alpha\wedge d\bar z_\beta, \qquad dz_\alpha\wedge d\bar z_\beta\wedge dz_\gamma\wedge d\bar z_\delta.
\]

Then, by reading Formulas \eqref{28} in complex coordinates, we get:

\begin{equation}\label{complex28}
\begin{aligned}
2 \psi_{12}&=(-\shortform{1\bar 2}+\shortform{2\bar 1}+\shortform{3 \bar 4}-\shortform{4\bar 3}+\shortform{5 \bar 6}-\shortform{6\bar 5}-\shortform{7 \bar 8}+\shortform{8\bar 7})-(-\shortform{1\bar 2}+\shortform{2\bar 1}+\shortform{3 \bar 4}-\shortform{4\bar 3}+\shortform{5 \bar 6}-\shortform{6\bar 5}-\shortform{7 \bar 8}+\shortform{8\bar 7})\shortform{'},
\\
2\psi_{13}&=(-\shortform{1\bar 3}+\shortform{3\bar 1}-\shortform{2 \bar 4}+\shortform{4\bar 2}+\shortform{5 \bar 7}-\shortform{7\bar 5}+\shortform{6 \bar 8}-\shortform{8\bar 6})-(-\shortform{1\bar 3}+\shortform{3\bar 1}-\shortform{2 \bar 4}+\shortform{4\bar 2}+\shortform{5 \bar 7}-\shortform{7\bar 5}+\shortform{6 \bar 8}-\shortform{8\bar 6})\shortform{'},
\\
2\psi_{14}&=(-\shortform{1\bar 4}+\shortform{4\bar 1}+\shortform{2\bar 3}-\shortform{3\bar 2}+\shortform{5\bar 8}-\shortform{8\bar 5}-\shortform{6\bar 7}+\shortform{7\bar 6})-(-\shortform{1\bar 4}+\shortform{4\bar 1}+\shortform{2\bar 3}-\shortform{3\bar 2}+\shortform{5\bar 8}-\shortform{8\bar 5}-\shortform{6\bar 7}+\shortform{7\bar 6})\shortform{'},
\\
2 \psi_{15}&=(-\shortform{1\bar 5}+\shortform{5\bar 1}-\shortform{2\bar 6}+\shortform{6\bar 2}-\shortform{3\bar 7}+\shortform{7\bar 3}-\shortform{4\bar 8}+\shortform{8\bar 4})-(-\shortform{1\bar 5}+\shortform{5\bar 1}-\shortform{2\bar 6}+\shortform{6\bar 2}-\shortform{3\bar 7}+\shortform{7\bar 3}-\shortform{4\bar 8}+\shortform{8\bar 4})\shortform{'},
\\
2\psi_{16}&=(-\shortform{1\bar 6}+\shortform{6\bar 1}+\shortform{2\bar 5}-\shortform{5\bar 2}-\shortform{3\bar 8}+\shortform{8\bar 3}+\shortform{4\bar 7}-\shortform{7\bar 4})-(-\shortform{1\bar 6}+\shortform{6\bar 1}+\shortform{2\bar 5}-\shortform{5\bar 2}-\shortform{3\bar 8}+\shortform{8\bar 3}+\shortform{4\bar 7}-\shortform{7\bar 4})\shortform{'},
\\
 2\psi_{17}&=(-\shortform{1\bar 7}+\shortform{7\bar 1}+\shortform{2\bar 8}-\shortform{8\bar 2}+\shortform{3\bar 5}-\shortform{5\bar 3}-\shortform{4\bar 6}+\shortform{6\bar 4})-(-\shortform{1\bar 7}+\shortform{7\bar 1}+\shortform{2\bar 8}-\shortform{8\bar 2}+\shortform{3\bar 5}-\shortform{5\bar 3}-\shortform{4\bar 6}+\shortform{6\bar 4})\shortform{'},
\\
2\psi_{18}&=(-\shortform{1\bar 8}+\shortform{8\bar 1}-\shortform{2\bar 7}+\shortform{7\bar 2}+\shortform{3\bar 6}-\shortform{6\bar 3}+\shortform{4\bar 5}-\shortform{5\bar 4})-(-\shortform{1\bar 8}+\shortform{8\bar 1}-\shortform{2\bar 7}+\shortform{7\bar 2}+\shortform{3\bar 6}-\shortform{6\bar 3}+\shortform{4\bar 5}-\shortform{5\bar 4})\shortform{'},
\\
2 \psi_{23}&=(-\shortform{1 \bar 4} + \shortform{4 \bar 1}+\shortform{2 \bar 3} - \shortform{3 \bar 2} -\shortform{5\bar 8}+\shortform{8\bar 5}+\shortform{6\bar 7}-\shortform{7\bar 6})+(-\shortform{1 \bar 4} + \shortform{4 \bar 1}+\shortform{2 \bar 3} - \shortform{3 \bar 2} -\shortform{5\bar 8}+\shortform{8\bar 5}+\shortform{6\bar 7}-\shortform{7\bar 6})\shortform{'},
\\
2 \psi_{24}&=(+\shortform{1\bar 3} - \shortform{3 \bar 1} +\shortform{2\bar 4} - \shortform{4 \bar 2}+\shortform{5\bar 7}-\shortform{7\bar 5}+\shortform{6\bar 8}-\shortform{8\bar 6})+(+\shortform{1\bar 3} - \shortform{3 \bar 1} +\shortform{2\bar 4} - \shortform{4 \bar 2}+\shortform{5\bar 7}-\shortform{7\bar 5}+\shortform{6\bar 8}-\shortform{8\bar 6})\shortform{'},
\\
2 \psi_{25}&=(-\shortform{1 \bar 6} + \shortform{6 \bar 1} +\shortform{2\bar 5} - \shortform{5 \bar 2}+\shortform{3\bar 8}-\shortform{8\bar 3}-\shortform{4\bar 7}+\shortform{7\bar 4})+(-\shortform{1 \bar 6} + \shortform{6 \bar 1} +\shortform{2\bar 5} - \shortform{5 \bar 2}+\shortform{3\bar 8}-\shortform{8\bar 3}-\shortform{4\bar 7}+\shortform{7\bar 4})\shortform{'},
\\
2 \psi_{26}&=(+\shortform{1 \bar 5}- \shortform{5 \bar 1} +\shortform{2\bar 6} - \shortform{6 \bar 2}-\shortform{3\bar 7}+\shortform{7\bar 3}-\shortform{4\bar 8}+\shortform{8\bar 4})+(+\shortform{1 \bar 5}- \shortform{5 \bar 1} +\shortform{2\bar 6} - \shortform{6 \bar 2}-\shortform{3\bar 7}+\shortform{7\bar 3}-\shortform{4\bar 8}+\shortform{8\bar 4})\shortform{'},
\\
2 \psi_{27}&=(+\shortform{1\bar 8}- \shortform{8 \bar 1} +\shortform{2\bar 7}- \shortform{7 \bar 2}+\shortform{3\bar 6}-\shortform{6\bar 3}+\shortform{4\bar 5}-\shortform{5\bar 4})+(+\shortform{1\bar 8}- \shortform{8 \bar 1} +\shortform{2\bar 7}- \shortform{7 \bar 2}+\shortform{3\bar 6}-\shortform{6\bar 3}+\shortform{4\bar 5}-\shortform{5\bar 4})\shortform{'},
\\
2 \psi_{28}&=(-\shortform{1\bar 7}+ \shortform{7 \bar 1} +\shortform{2\bar 8}- \shortform{8 \bar 2}-\shortform{3\bar 5}+\shortform{5\bar 3}+\shortform{4\bar 6}-\shortform{6\bar 4})+(-\shortform{1\bar 7}+ \shortform{7 \bar 1} +\shortform{2\bar 8}- \shortform{8 \bar 2}-\shortform{3\bar 5}+\shortform{5\bar 3}+\shortform{4\bar 6}-\shortform{6\bar 4})\shortform{'},
\\
2 \psi_{34}&=(-\shortform{1 \bar 2}+ \shortform{2 \bar 1} +\shortform{3\bar 4}- \shortform{4 \bar 3}-\shortform{5\bar 6}+\shortform{6\bar 5}+\shortform{7\bar 8}-\shortform{8\bar 7})+(-\shortform{1 \bar 2}+ \shortform{2 \bar 1} +\shortform{3\bar 4}- \shortform{4 \bar 3}-\shortform{5\bar 6}+\shortform{6\bar 5}+\shortform{7\bar 8}-\shortform{8\bar 7})\shortform{'},
\\
2 \psi_{35}&=(-\shortform{1 \bar  7}+ \shortform{7 \bar 1} -\shortform{2\bar 8}+ \shortform{8 \bar 2}+\shortform{3\bar 5}-\shortform{5\bar 3}+\shortform{4\bar 6}-\shortform{6\bar 4})+(-\shortform{1 \bar  7}+ \shortform{7 \bar 1} -\shortform{2\bar 8}+ \shortform{8 \bar 2}+\shortform{3\bar 5}-\shortform{5\bar 3}+\shortform{4\bar 6}-\shortform{6\bar 4})\shortform{'},
\\
2 \psi_{36}&=(-\shortform{1 \bar 8}+ \shortform{8 \bar 1} +\shortform{2\bar 7}- \shortform{7 \bar 2}+\shortform{3\bar 6}-\shortform{6\bar 3}-\shortform{4\bar 5}+\shortform{5\bar 4})+(-\shortform{1 \bar 8}+ \shortform{8 \bar 1} +\shortform{2\bar 7}- \shortform{7 \bar 2}+\shortform{3\bar 6}-\shortform{6\bar 3}-\shortform{4\bar 5}+\shortform{5\bar 4})\shortform{'},
\\
2 \psi_{37}&=(+\shortform{1 \bar 5}- \shortform{5\bar 1} -\shortform{2\bar 6}+ \shortform{6 \bar 2}+\shortform{3\bar 7}-\shortform{7\bar 3}-\shortform{4\bar 8}+\shortform{8\bar 4})+(+\shortform{1 \bar 5}- \shortform{5\bar 1} -\shortform{2\bar 6}+ \shortform{6 \bar 2}+\shortform{3\bar 7}-\shortform{7\bar 3}-\shortform{4\bar 8}+\shortform{8\bar 4})\shortform{'},
\\
2 \psi_{38}&=(+\shortform{1 \bar 6}- \shortform{6 \bar 1} +\shortform{2\bar 5}-\shortform{5 \bar 2}+\shortform{3\bar 8}-\shortform{8\bar 3}+\shortform{4\bar 7}-\shortform{7\bar 4})+(+\shortform{1 \bar 6}- \shortform{6 \bar 1} +\shortform{2\bar 5}-\shortform{5 \bar 2}+\shortform{3\bar 8}-\shortform{8\bar 3}+\shortform{4\bar 7}-\shortform{7\bar 4})\shortform{'},
\\
2 \psi_{45}&=(-\shortform{1 \bar 8}+ \shortform{8 \bar 1} +\shortform{2\bar 7}-\shortform{7 \bar 2}-\shortform{3\bar 6}+\shortform{6\bar 3}+\shortform{4\bar 5}-\shortform{5\bar 4})+(-\shortform{1 \bar 8}+ \shortform{8 \bar 1} +\shortform{2\bar 7}-\shortform{7 \bar 2}-\shortform{3\bar 6}+\shortform{6\bar 3}+\shortform{4\bar 5}-\shortform{5\bar 4})\shortform{'},
\\
2 \psi_{46}&=(+\shortform{1 \bar 7}- \shortform{7 \bar 1} +\shortform{2\bar 8}-\shortform{8 \bar 2}+\shortform{3\bar 5}-\shortform{5\bar 3}+\shortform{4\bar 6}-\shortform{6\bar 4})+(+\shortform{1 \bar 7}- \shortform{7 \bar 1} +\shortform{2\bar 8}-\shortform{8 \bar 2}+\shortform{3\bar 5}-\shortform{5\bar 3}+\shortform{4\bar 6}-\shortform{6\bar 4})\shortform{'},
\\
2 \psi_{47}&=(-\shortform{1 \bar 6}+ \shortform{6 \bar 1} -\shortform{2\bar 5}+ \shortform{5 \bar 2}+\shortform{3\bar 8}-\shortform{8\bar 3}+\shortform{4\bar 7}-\shortform{7\bar 4})+(-\shortform{1 \bar 6}+ \shortform{6 \bar 1} -\shortform{2\bar 5}+ \shortform{5 \bar 2}+\shortform{3\bar 8}-\shortform{8\bar 3}+\shortform{4\bar 7}-\shortform{7\bar 4})\shortform{'},
\\
2 \psi_{48}&=(+\shortform{1 \bar 5}- \shortform{5 \bar 1} -\shortform{2\bar 6}+\shortform{6 \bar 2}-\shortform{3\bar 7}+\shortform{7\bar 3}+\shortform{4\bar 8}-\shortform{8\bar 4})+(+\shortform{1 \bar 5}- \shortform{5 \bar 1} -\shortform{2\bar 6}+\shortform{6 \bar 2}-\shortform{3\bar 7}+\shortform{7\bar 3}+\shortform{4\bar 8}-\shortform{8\bar 4})\shortform{'},
\\
2 \psi_{56}&=(-\shortform{1 \bar 2}+ \shortform{2 \bar 1} -\shortform{3\bar 4}+\shortform{4 \bar 3}+\shortform{5\bar 6}-\shortform{6\bar 5}+\shortform{7\bar 8}-\shortform{8\bar 7})+(-\shortform{1 \bar 2}+ \shortform{2 \bar 1} -\shortform{3\bar 4}+\shortform{4 \bar 3}+\shortform{5\bar 6}-\shortform{6\bar 5}+\shortform{7\bar 8}-\shortform{8\bar 7})\shortform{'},
\\
2 \psi_{57}&=(-\shortform{1 \bar 3}+ \shortform{3 \bar 1} +\shortform{2\bar 4} - \shortform{4 \bar 2}+\shortform{5\bar 7}-\shortform{7\bar 5}-\shortform{6\bar 8}+\shortform{8\bar 6})+(-\shortform{1 \bar 3}+ \shortform{3 \bar 1} +\shortform{2\bar 4} - \shortform{4 \bar 2}+\shortform{5\bar 7}-\shortform{7\bar 5}-\shortform{6\bar 8}+\shortform{8\bar 6})\shortform{'},
\\
2 \psi_{58}&=(-\shortform{1 \bar 4}+ \shortform{4 \bar 1} -\shortform{2\bar 3}+\shortform{3 \bar 2}+\shortform{5\bar 8}-\shortform{8\bar 5}+\shortform{6\bar 7}-\shortform{7\bar 6})+(-\shortform{1 \bar 4}+ \shortform{4 \bar 1} -\shortform{2\bar 3}+\shortform{3 \bar 2}+\shortform{5\bar 8}-\shortform{8\bar 5}+\shortform{6\bar 7}-\shortform{7\bar 6})\shortform{'},
\\
2 \psi_{67}&=(+\shortform{1 \bar 4}- \shortform{4 \bar 1} +\shortform{2\bar 3}-\shortform{3 \bar 2}+\shortform{5\bar 8}-\shortform{8\bar 5}+\shortform{6\bar 7}-\shortform{7\bar 6})+(+\shortform{1 \bar 4}- \shortform{4 \bar 1} +\shortform{2\bar 3}-\shortform{3 \bar 2}+\shortform{5\bar 8}-\shortform{8\bar 5}+\shortform{6\bar 7}-\shortform{7\bar 6})\shortform{'},
\\
2 \psi_{68}&=(-\shortform{1 \bar 3}+ \shortform{3 \bar 1} +\shortform{2\bar 4}-\shortform{4 \bar 2}-\shortform{5\bar 7}+\shortform{7\bar 5}+\shortform{6\bar 8}-\shortform{8\bar 6})+(-\shortform{1 \bar 3}+ \shortform{3 \bar 1} +\shortform{2\bar 4}-\shortform{4 \bar 2}-\shortform{5\bar 7}+\shortform{7\bar 5}+\shortform{6\bar 8}-\shortform{8\bar 6})\shortform{'},
\\
2 \psi_{78}&=(+\shortform{1 \bar 2}- \shortform{2 \bar 1} +\shortform{3\bar 4}-\shortform{4 \bar 3}+\shortform{5\bar 6}-\shortform{6\bar 5}+\shortform{7\bar 8}-\shortform{8\bar 7})+(+\shortform{1 \bar 2}- \shortform{2 \bar 1} +\shortform{3\bar 4}-\shortform{4 \bar 3}+\shortform{5\bar 6}-\shortform{6\bar 5}+\shortform{7\bar 8}-\shortform{8\bar 7})\shortform{'},
\end{aligned}
\end{equation}
and, from Formulas \eqref{8}:
\begin{equation}\label{complex8}
\begin{aligned}
2\psi_{19}&=(-\shortform{1\bar 1'}+\shortform{1' \bar 1}-\shortform{2\bar 2'}+\shortform{2'\bar 2} -\shortform{3\bar 3'}+\shortform{3'\bar 3} -\shortform{4\bar 4'}+\shortform{4'\bar 4} -\shortform{5\bar 5'}+\shortform{5'\bar 5}-\shortform{6\bar 6'}+\shortform{6'\bar 6}-\shortform{7\bar 7'}+\shortform{7'\bar 7}-\shortform{8\bar 8'}+\shortform{8'\bar 8}),
\\
2\psi_{29}&=(-\shortform{1\bar 2'}+\shortform{2' \bar 1}+\shortform{2\bar 1'}-\shortform{1'\bar 2} +\shortform{3\bar 4'}-\shortform{4'\bar 3} -\shortform{4\bar 3'}+\shortform{3'\bar 4} +\shortform{5\bar 6'}-\shortform{6'\bar 5}-\shortform{6\bar 5'}+\shortform{5'\bar 6}-\shortform{7\bar 8'}+\shortform{8'\bar 7}+\shortform{8\bar 7'}-\shortform{7'\bar 8}),
\\
2\psi_{39}&=(-\shortform{1\bar 3'}+\shortform{3' \bar 1}-\shortform{2\bar 4'}+\shortform{4'\bar 2} +\shortform{3\bar 1'}-\shortform{1'\bar 3} +\shortform{4\bar 2'}-\shortform{2'\bar 4} +\shortform{5\bar 7'}-\shortform{7'\bar 5}+\shortform{6\bar 8'}-\shortform{8'\bar 6}-\shortform{7\bar 5'}+\shortform{5'\bar 7}-\shortform{8\bar 6'}+\shortform{6'\bar 8}),
\\
2\psi_{49}&=(-\shortform{1\bar 4'}+\shortform{4' \bar 1}+\shortform{2\bar 3'}-\shortform{3'\bar 2} -\shortform{3\bar 2'}+\shortform{2'\bar 3} +\shortform{4\bar 1'}-\shortform{1'\bar 4} +\shortform{5\bar 8'}-\shortform{8'\bar 5}-\shortform{6\bar 7'}+\shortform{7'\bar 6}+\shortform{7\bar 6'}-\shortform{6'\bar 7}-\shortform{8\bar 5'}+\shortform{5'\bar 8}),
\\
2\psi_{59}&=(-\shortform{1\bar 5'}+\shortform{5' \bar 1}-\shortform{2\bar 6'}+\shortform{6'\bar 2} -\shortform{3\bar 7'}+\shortform{7'\bar 3} -\shortform{4\bar 8'}+\shortform{8'\bar 4} +\shortform{5\bar 1'}-\shortform{1'\bar 5}+\shortform{6\bar 2'}-\shortform{2'\bar 6}+\shortform{7\bar 3'}-\shortform{3'\bar 7}+\shortform{8\bar 4'}-\shortform{4'\bar 8}),
\\
2\psi_{69}&=(-\shortform{1\bar 6'}+\shortform{6' \bar 1}+\shortform{2\bar 5'}-\shortform{5'\bar 2} -\shortform{3\bar 8'}+\shortform{8'\bar 3} +\shortform{4\bar 7'}-\shortform{7'\bar 4} -\shortform{5\bar 2'}+\shortform{2'\bar 5}+\shortform{6\bar 1'}-\shortform{1'\bar 6}-\shortform{7\bar 4'}+\shortform{4'\bar 7}+\shortform{8\bar 3'}-\shortform{3'\bar 8}),
\\
2\psi_{79}&=(-\shortform{1\bar 7'}+\shortform{7' \bar 1}+\shortform{2\bar 8'}-\shortform{8'\bar 2} +\shortform{3\bar 5'}-\shortform{5'\bar 3} -\shortform{4\bar 6'}+\shortform{6'\bar 4} -\shortform{5\bar 3'}+\shortform{3'\bar 5}+\shortform{6\bar 4'}-\shortform{4'\bar 6}+\shortform{7\bar 1'}-\shortform{1'\bar 7}-\shortform{8\bar 2'}+\shortform{2'\bar 8}),
\\
2\psi_{89}&=(-\shortform{1\bar 8'}+\shortform{8' \bar 1}-\shortform{2\bar 7'}+\shortform{7'\bar 2} +\shortform{3\bar 6'}-\shortform{6'\bar 3} +\shortform{4\bar 5'}-\shortform{5'\bar 4} -\shortform{5\bar 4'}+\shortform{4'\bar 5}-\shortform{6\bar 3'}+\shortform{3'\bar 6}+\shortform{7\bar 2'}-\shortform{2'\bar 7}+\shortform{8\bar 1'}-\shortform{1'\bar 8}).
\end{aligned}
\end{equation}

The ``new'' K\"ahler forms $\psi_{01}, \psi_{02}, \dots, \psi_{09}$, associated with $J_{01}, J_{02}, \dots, J_{09}$, read:

\begin{equation}\label{complex9}
\begin{aligned}
2\psi_{01}&=i(+\shortform{1\bar 1'}+\shortform{1' \bar 1}+\shortform{2\bar 2'}+\shortform{2'\bar 2} +\shortform{3\bar 3'}+\shortform{3'\bar 3} +\shortform{4\bar 4'}+\shortform{4'\bar 4} +\shortform{5\bar 5'}+\shortform{5'\bar 5}+\shortform{6\bar 6'}+\shortform{6'\bar 6}+\shortform{7\bar 7'}+\shortform{7'\bar 7}+\shortform{8\bar 8'}+\shortform{8'\bar 8}),
\\
2\psi_{02}&=i(+\shortform{1\bar 2'}+\shortform{2' \bar 1}-\shortform{2\bar 1'}-\shortform{1'\bar 2} -\shortform{3\bar 4'}-\shortform{4'\bar 3} +\shortform{4\bar 3'}+\shortform{3'\bar 4} -\shortform{5\bar 6'}-\shortform{6'\bar 5}+\shortform{6\bar 5'}+\shortform{5'\bar 6}+\shortform{7\bar 8'}+\shortform{8'\bar 7}-\shortform{8\bar 7'}-\shortform{7'\bar 8}),
\\
2\psi_{03}&=i(+\shortform{1\bar 3'}+\shortform{3' \bar 1}+\shortform{2\bar 4'}+\shortform{4'\bar 2} -\shortform{3\bar 1'}-\shortform{1'\bar 3} -\shortform{4\bar 2'}-\shortform{2'\bar 4} -\shortform{5\bar 7'}-\shortform{7'\bar 5}-\shortform{6\bar 8'}-\shortform{8'\bar 6}+\shortform{7\bar 5'}+\shortform{5'\bar 7}+\shortform{8\bar 6'}+\shortform{6'\bar 8}),
\\
2\psi_{04}&=i(+\shortform{1\bar 4'}+\shortform{4' \bar 1}-\shortform{2\bar 3'}-\shortform{3'\bar 2} +\shortform{3\bar 2'}+\shortform{2'\bar 3} -\shortform{4\bar 1'}-\shortform{1'\bar 4} -\shortform{5\bar 8'}-\shortform{8'\bar 5}+\shortform{6\bar 7'}+\shortform{7'\bar 6}-\shortform{7\bar 6'}-\shortform{6'\bar 7}+\shortform{8\bar 5'}+\shortform{5'\bar 8}),
\\
2\psi_{05}&=i(-\shortform{1\bar 5'}-\shortform{5' \bar 1}-\shortform{2\bar 6'}-\shortform{6'\bar 2} +\shortform{3\bar 7'}+\shortform{7'\bar 3} +\shortform{4\bar 8'}+\shortform{8'\bar 4} +\shortform{5\bar 1'}+\shortform{1'\bar 5}+\shortform{6\bar 2'}+\shortform{2'\bar 6}-\shortform{7\bar 3'}-\shortform{3'\bar 7}-\shortform{8\bar 4'}-\shortform{4'\bar 8}),
\\
2\psi_{06}&=i(-\shortform{1\bar 6'}-\shortform{6' \bar 1}+\shortform{2\bar 5'}+\shortform{5'\bar 2} -\shortform{3\bar 8'}-\shortform{8'\bar 3} +\shortform{4\bar 7'}+\shortform{7'\bar 4} -\shortform{5\bar 2'}-\shortform{2'\bar 5}+\shortform{6\bar 1'}+\shortform{1'\bar 6}-\shortform{7\bar 4'}-\shortform{4'\bar 7}+\shortform{8\bar 3'}+\shortform{3'\bar 8}),
\\
2\psi_{07}&=i(-\shortform{1\bar 7'}-\shortform{7' \bar 1}+\shortform{2\bar 8'}+\shortform{8'\bar 2} +\shortform{3\bar 5'}+\shortform{5'\bar 3} -\shortform{4\bar 6'}-\shortform{6'\bar 4} -\shortform{5\bar 3'}-\shortform{3'\bar 5}+\shortform{6\bar 4'}+\shortform{4'\bar 6}+\shortform{7\bar 1'}+\shortform{1'\bar 7}-\shortform{8\bar 2'}-\shortform{2'\bar 8}),
\\
2\psi_{08}&=i(-\shortform{1\bar 8'}-\shortform{8' \bar 1}-\shortform{2\bar 7'}-\shortform{7'\bar 2} +\shortform{3\bar 6'}+\shortform{6'\bar 3} +\shortform{4\bar 5'}+\shortform{5'\bar 4} -\shortform{5\bar 4'}-\shortform{4'\bar 5}-\shortform{6\bar 3'}-\shortform{3'\bar 6}+\shortform{7\bar 2'}+\shortform{2'\bar 7}+\shortform{8\bar 1'}+\shortform{1'\bar 8})
,
\\
2\psi_{09}&=i(+\shortform{1\bar 1}+\shortform{2\bar 2}+\shortform{3\bar 3} +\shortform{4\bar 4}+\shortform{5\bar 5}+\shortform{6\bar 6}+\shortform{7\bar 7}+\shortform{8\bar 8} -\shortform{1'\bar 1'}-\shortform{2'\bar 2'}-\shortform{3'\bar 3'} -\shortform{4'\bar 4'}-\shortform{5'\bar 5'}-\shortform{6'\bar 6'}-\shortform{7'\bar 7'}-\shortform{8'\bar 8'}). \end{aligned}
\end{equation}

Then:

\begin{pr}\label{iso}
The Lie subalgebras $\mathfrak{h}$ and $\liespin{10}$ of $\lieso{32}$ are isomorphic.
\end{pr}

\begin{proof} An isomorphism can be defined through the choices of our bases.  Just look at the correspondence
\[
P_{\alpha\beta} \longleftrightarrow J_{\alpha +1, \beta +1}, \; \; 0\leq \alpha<\beta \leq 8, \qquad P_{\alpha9} \longleftrightarrow J_{0, \alpha +1}, \; \; 0\leq \alpha \leq 8.
\]
\end{proof}

\emph{Beginning of the Proof of Theorem \ref{half-spin}.}
The two isomorphic Lie subalgebras $\mathfrak{h}, \liespin{10} \subset \lieso{32}$ correspond to two subgroups $H, H'\subset \SO{32}$. Note that both $H$ and $H'$ are isomorphic to $\Spin{10}$. For the subgroup $H$, this is recognized by the characterization of $\Spin{n}$ as the group generated by unit bivectors in the multiplicative group of invertible element in the ambient Clifford algebra (see for example \cite[page 198]{HarSpC}). As for $H'$, its isomorphism with $\Spin{10}$ is a consequence of how we constructed the group $H' \subset \SU{16}$ and its Lie algebra at the beginning of this Section. Note that in the previous discussion we already denoted by $\Spin{10}$ the subgroup $H'\subset \SU{16}$  and by $\liespin{10}$ its Lie algebra. By comparing with the half-spin representation theory (\cite[Chapter 3]{MeiCAL}, in particular pages 80--85), it follows that $H$ is a real non-spin representation of $\Spin{10}$ and that $H'$ is the image under one of the two non-isomorphic and conjugate half-spin representations of the abstract group $\Spin{10}$.
\hfill $\Box$

\begin{re} \rm{Concerning the subgroup $H \subset\SO{32}$, observe that among the endomorphisms $\PP_0, \PP_1,\dots , \PP_9$ of $\RR^{32}$ defined by Formulas \eqref{top32}, just $P_6$ ad $P_7$ are not in $\U{16} \subset \SO{32}$. A detailed description of them gives:
\[
\scriptsize{\PP_6= \left(\begin{array}{c|c|c|c}
0& 0&-R_g&0 \\
\hline
0& 0&0&R_g \\
\hline
R_g& 0&0&0 \\
\hline
0&-R_g&0&0
\end{array}
\right),
\qquad\PP_7 = \left(
\begin{array}{c|c|c|c}
0& 0&-R_h&0 \\
\hline
0& 0&0&R_h \\
\hline
R_h& 0&0&0 \\
\hline
0&-R_h&0&0
\end{array}
\right),}
\]
where
\[
\scriptsize{R_g= \left(\begin{array}{c|c}
0&L^\HH_j \\
\hline
L^\HH_j& 0
\end{array}
\right), \qquad R_h= \left(\begin{array}{c|c}
0&L^\HH_k \\
\hline
L^\HH_k& 0
\end{array}
\right)
,}
\]
and the left quaternionic matrices
\[
\scriptsize{L^\HH_j= \left(\begin{array}{cccc}
0&0&-1&0 \\
0&0&0&1 \\
1&0&0&0\\
0&-1&0&0
\end{array}
\right), \qquad L^\HH_k= \left(\begin{array}{cccc}
0&0&0&-1 \\
0&0&-1&0 \\
0&1&0&0\\
1&0&0&0
\end{array}
\right)
}
\]
are not in the subgroup $\U{2} \subset \SO{4}$ (cf.\ \cite[pages 329 and 331]{PaPSAC}).}
\end{re}


\section{The canonical 8-form $\Phi_{\Spin{10}}$}

Look now at the skew-symmetric matrix $\psi^D= (\psi_{\alpha \beta})_{0 \leq \alpha, \beta \leq 9}$ of the K\"ahler forms defined for $\alpha <\beta$ in the previous Section and associated with the family $J^D$. The skew-symmetry of $\psi^D$ as matrix of K\"ahler forms associated with complex structures is insured by setting (formally, coherently with Proposition \ref{iso} and for $\alpha=1, \dots ,9$):
\[
J_{0 \alpha} = i \wedge \I_\alpha = -\I_\alpha \wedge i = -J_{\alpha 0}.
\]
Denote by $\tau_2 = \sum_{0 \leq \alpha < \beta \leq 9} \psi^2_{\alpha \beta}$ the second coefficient of its characteristic polynomial.

\begin{te}\label{tau}
\[
\tau_2 (\psi^D)= - 3\omega^2,
\]
where
\[
\omega = \frac{i}{2}(\shortform{1\bar 1}+ \dots + \shortform{8\bar 8} + \shortform{1' \bar 1'} + \dots + \shortform{8' \bar 8'})
\]
is the K\"ahler 2-form of  the complex structure $\mathcal I$ on $\CC^{16}$.
\end{te}

\begin{proof}

Decompose $\tau_2$ as follows:
\[
\tau_2 = \rho_2 + \mu_2 + \nu_2,
\]
where
\[
\rho_2 =  \sum_{1 \leq \alpha < \beta \leq 8} \psi^2_{\alpha \beta},\quad \mu_2 =  \sum_{1 \leq \gamma \leq 8} \psi^2_{\gamma 9},  \quad \nu_2=  \sum_{1 \leq \delta \leq 9} \psi^2_{0 \delta}
\]
and note that the 2-forms $\psi$ appearing in the three sums are listed in \eqref{complex28}, \eqref{complex8}, \eqref{complex9}, respectively.

Look first at the restrictions $\rho_2 \vert_V, \mu_2 \vert_V,\nu_2 \vert_V$, $\rho_2 \vert_{V'},\mu_2 \vert_{V'},\nu_2 \vert_{V'}$ to the subspaces
\[
V = <\shortform{1,2,\dots,8}>, \qquad V'=<\shortform{1',2',\dots,8'}>.
\]

We get:
\[
\rho_2 \vert_V = 2(\shortform{1\bar 1 2 \bar 2+ \dots +7\bar 7 8 \bar 8}) = -4(\omega \vert_V)^2, \qquad \rho_2 \vert_{V'} = 2(\shortform{1'\bar 1' 2' \bar 2'+ \dots +7'\bar 7' 8' \bar 8'}) = -4(\omega \vert_{V'})^2,
\]
\[
\mu_2 \vert_V = \nu_2 \vert_V = \mu_2 \vert_{V'} =\nu_2 \vert_{V'} =0.
\]

Decompose now $\rho_2$ as
\[
\rho_2 = \rho_2 \vert_V + \rho_2 \vert_{V'} + \tilde \rho_2,
\]
and to compute $\tilde \rho_2$ observe first that the K\"ahler forms $\psi_{\alpha \beta}$ in \eqref{complex28} can be arranged in the following seven families:
\begin{equation}\label{[]}
\begin{split}
\psi_{12},\psi_{34},\psi_{56}, \psi_{78}&=(\pm[\shortform{12}]\pm[\shortform{34}]\pm[\shortform{56}]\pm[\shortform{78}])\pm(\pm[\shortform{12}]\shortform{'}\pm[\shortform{34}]\shortform{'}\pm[\shortform{56}]\shortform{'}\pm[\shortform{78}]\shortform{'}),
\\
\psi_{13},\psi_{24},\psi_{57},\psi_{68}&=(\pm[\shortform{13}]\pm[\shortform{24}]\pm[\shortform{57}]\pm[\shortform{68}])\pm(\pm[\shortform{13}]\shortform{'}\pm[\shortform{24}]\shortform{'}\pm[\shortform{57}]\shortform{'}\pm[\shortform{68}]\shortform{'}),
\\
\psi_{14},\psi_{23},\psi_{58},\psi_{67}&=(\pm[\shortform{14}]\pm[\shortform{23}]\pm[\shortform{58}]\pm[\shortform{67}])\pm(\pm[\shortform{14}]\shortform{'}\pm[\shortform{23}]\shortform{'}\pm[\shortform{58}]\shortform{'}\pm[\shortform{67}]\shortform{'}),
\\
\psi_{15},\psi_{26},\psi_{37},\psi_{48}&=(\pm[\shortform{15}]\pm[\shortform{26}]\pm[\shortform{37}]\pm[\shortform{48}])\pm(\pm[\shortform{15}]\shortform{'}\pm[\shortform{26}]\shortform{'}\pm[\shortform{37}]\shortform{'}\pm[\shortform{48}]\shortform{'}),
\\
\psi_{16},\psi_{25},\psi_{38},\psi_{47}&=(\pm[\shortform{16}]\pm[\shortform{25}]\pm[\shortform{38}]\pm[\shortform{47}])\pm(\pm[\shortform{16}]\shortform{'}\pm[\shortform{25}]\shortform{'}\pm[\shortform{38}]\shortform{'}\pm[\shortform{47}]\shortform{'}),
\\
\psi_{17},\psi_{28},\psi_{35},\psi_{46}&=(\pm[\shortform{17}]\pm[\shortform{28}]\pm[\shortform{35}]\pm[\shortform{46}])\pm(\pm[\shortform{17}]\shortform{'}\pm[\shortform{28}]\shortform{'}\pm[\shortform{35}]\shortform{'}\pm[\shortform{46}]\shortform{'}),
\\
\psi_{18},\psi_{27},\psi_{36},\psi_{45}&=(\pm[\shortform{18}]\pm[\shortform{27}]\pm[\shortform{36}]\pm[\shortform{45}])\pm(\pm[\shortform{18}]\shortform{'}\pm[\shortform{27}]\shortform{'}\pm[\shortform{36}]\shortform{'}\pm[\shortform{45}]\shortform{'}),
\end{split}
\end{equation}
where we use short notations like for example
\[
\pm[\shortform{12}] = \pm (-\shortform{1\bar 2}+\shortform{2\bar 1}).
\]
In each line of Formulas \eqref{[]} the signs before brackets follow (up to a global change) the four patterns
\begin{equation}\label{pm}
\begin{split}
-&[\phantom{\shortform{12}}]+[\phantom{\shortform{12}}]+[\phantom{\shortform{12}}]-[\phantom{\shortform{12}}],\qquad -[\phantom{\shortform{12}}]+[\phantom{\shortform{12}}]-[\phantom{\shortform{12}}]+[\phantom{\shortform{12}}],
\\
-&[\phantom{\shortform{12}}]-[\phantom{\shortform{12}}]+[\phantom{\shortform{12}}]+[\phantom{\shortform{12}}],\qquad +[\phantom{\shortform{12}}]+[\phantom{\shortform{12}}]+[\phantom{\shortform{12}}]+[\phantom{\shortform{12}}],
\end{split}
\end{equation}
that is, an even number of $+$ and of $-$ signs. Also, in all these seven families (and again up to a global change for forms of type $\psi_{1 \beta}$, for $\beta=2,\dots,8$), the same pattern appears both in the terms with coordinates $(1,\dots,8)$ and in those with coordinates $(1',\dots,8')$.

These observations allow to compute $\tilde \rho_2$. Namely, the $\tilde \rho_2$ components $\widetilde{( \dots  \dots)}$ of the sums of four squared K\"ahler forms listed in each line of Formulas \eqref{[]} read:

\begin{equation}\label{7sums}
\begin{split}
\widetilde{(\psi^2_{12}+\psi^2_{34}+\psi^2_{56}+\psi^2_{78}}&)=[\shortform{1\bar21'\bar2'}+\shortform{1\bar23'\bar4'}+\shortform{1\bar25'\bar6'}-\shortform{1\bar27'\bar8'}+\shortform{3\bar41'\bar2'}+ \shortform{3\bar43'\bar4'}-\shortform{3\bar45'\bar6'}+\shortform{3\bar47'\bar8'}
\\
&+\shortform{5\bar61'\bar2'}-\shortform{5\bar63'\bar4'}+\shortform{5\bar65'\bar6'}+\shortform{5\bar67'\bar8'}-\shortform{7\bar81'\bar2'}+\shortform{7\bar83'\bar4'}+\shortform{7\bar85'\bar6'}+\shortform{7\bar87'\bar8'}]_{\text{odd-odd}} \enspace +
\\
&[-\shortform{1\bar22'\bar1'}-\shortform{1\bar24'\bar3'}-\shortform{1\bar26'\bar5'}+\shortform{1\bar28'\bar7'}-\shortform{3\bar42'\bar1'}-\shortform{3\bar44'\bar3'}+\shortform{3\bar46'\bar5'}-\shortform{3\bar48'\bar7'}
\\
&-\shortform{5\bar62'\bar1'}+\shortform{5\bar64'\bar3'}-\shortform{5\bar66'\bar5'}-\shortform{5\bar68'\bar7'}+\shortform{7\bar82'\bar1'}-\shortform{7\bar84'\bar3'}-\shortform{7\bar86'\bar5'}-\shortform{7\bar88'\bar7'}]_{\text{odd-even}} \enspace +
\\
&[-\shortform{2\bar11'\bar2'}-\shortform{2\bar13'\bar4'}-\shortform{2\bar15'\bar6'}+\shortform{2\bar17'\bar8'}-\shortform{4\bar31'\bar2'}-\shortform{4\bar33'\bar4'}+\shortform{4\bar35'\bar6'}-\shortform{4\bar37'\bar8'}
\\
&-\shortform{6\bar51'\bar2'}+\shortform{6\bar53'\bar4'}-\shortform{6\bar55'\bar6'}-\shortform{6\bar57'\bar8'}+\shortform{8\bar71'\bar2'}-\shortform{8\bar73'\bar4'}-\shortform{8\bar75'\bar6'}-\shortform{8\bar77'\bar8'}]_{\text{even-odd}} \enspace +
\\
&[\shortform{2\bar12'\bar1'}+\shortform{2\bar14'\bar3'}+\shortform{2\bar16'\bar5'}-\shortform{2\bar18'\bar7'}+\shortform{4\bar32'\bar1'}+\shortform{4\bar34'\bar3'}-\shortform{4\bar36'\bar5'}+\shortform{4\bar38'\bar7'}
\\
&+\shortform{6\bar52'\bar1'}-\shortform{6\bar54'\bar3'}+\shortform{6\bar56'\bar5'}+\shortform{6\bar58'\bar7'}-\shortform{8\bar72'\bar1'}+\shortform{8\bar74'\bar3'}+\shortform{8\bar76'\bar5'}+\shortform{8\bar78'\bar7'}]_{\text{even-even}} \enspace ,
\\
\widetilde{(\psi^2_{13}+\psi^2_{24}+\psi^2_{57}+\psi^2_{68}}&)=[\shortform{1\bar31'\bar3'}-\shortform{1\bar32'\bar4'}+\shortform{1\bar35'\bar7'}+\shortform{1\bar36'\bar8'}-\shortform{2\bar41'\bar3'}+\shortform{2\bar42'\bar4'}+\shortform{2\bar45'\bar7'}+\shortform{2\bar46'\bar8'}
\\
&+\shortform{5\bar71'\bar3'}+\shortform{5\bar72'\bar4'}+\shortform{5\bar75'\bar7'}-\shortform{5\bar76'\bar8'}+\shortform{6\bar81'\bar3'}+\shortform{6\bar82'\bar4'}-\shortform{6\bar85'\bar7'}+\shortform{6\bar86'\bar8'}]_{\text{odd-odd}} \enspace +
\\
&[\dots]_{\text{odd-even}} \enspace + [\dots]_{\text{even-odd}} \enspace + [\dots]_{\text{even-even}} \enspace ,
\\
\widetilde{(\psi^2_{14}+\psi^2_{23}+\psi^2_{58}+\psi^2_{67}}&)=[\shortform{1\bar41'\bar4'}+\shortform{1\bar42'\bar3'}+\shortform{1\bar45'\bar8'}-\shortform{1\bar46'\bar7'}+\shortform{2\bar31'\bar4'}+\shortform{2\bar32'\bar3'}-\shortform{2\bar35'\bar8'}+\shortform{2\bar36'\bar7'}
\\
&+\shortform{5\bar81'\bar4'}-\shortform{5\bar82'\bar3'}+\shortform{5\bar85'\bar8'}+\shortform{5\bar86'\bar7'}-\shortform{6\bar71'\bar4'}+\shortform{6\bar72'\bar3'}+\shortform{6\bar75'\bar8'}+\shortform{6\bar76'\bar7'}]_{\text{odd-odd}} \enspace +
\\
&[\dots]_{\text{odd-even}} \enspace + [\dots]_{\text{even-odd}} \enspace + [\dots]_{\text{even-even}} \enspace ,
\\
\widetilde{(\psi^2_{15}+\psi^2_{26}+\psi^2_{37}+\psi^2_{48}}&)=[\shortform{1\bar51'\bar5'}-\shortform{1\bar52'\bar6'}-\shortform{1\bar53'\bar7'}-\shortform{1\bar54'\bar8'}-\shortform{2\bar61'\bar5'}+\shortform{2\bar62'\bar6'}-\shortform{2\bar63'\bar7'}-\shortform{2\bar64'\bar8'}
\\
&-\shortform{3\bar71'\bar5'}-\shortform{3\bar72'\bar6'}+\shortform{3\bar73'\bar7'}-\shortform{3\bar74'\bar8'}-\shortform{4\bar81'\bar5'}-\shortform{4\bar82'\bar6'}-\shortform{4\bar83'\bar7'}+\shortform{4\bar84'\bar8'}]_{\text{odd-odd}} \enspace +
\\
&[\dots]_{\text{odd-even}} \enspace + [\dots]_{\text{even-odd}} \enspace + [\dots]_{\text{even-even}} \enspace ,
\\
\widetilde{(\psi^2_{16}+\psi^2_{25}+\psi^2_{38}+\psi^2_{47}}&)=[\shortform{1\bar61'\bar6'}+\shortform{1\bar62'\bar5'}-\shortform{1\bar63'\bar8'}+\shortform{1\bar64'\bar7'}+\shortform{2\bar51'\bar6'}+\shortform{2\bar52'\bar5'}+\shortform{2\bar53'\bar8'}-\shortform{2\bar54'\bar7'}
\\
&-\shortform{3\bar81'\bar6'}+\shortform{3\bar82'\bar5'}+\shortform{3\bar83'\bar8'}+\shortform{3\bar84'\bar7'}+\shortform{4\bar71'\bar6'}-\shortform{4\bar72'\bar5'}+\shortform{4\bar73'\bar8'}+\shortform{4\bar74'\bar7'}]_{\text{odd-odd}} \enspace +
\\
&[\dots]_{\text{odd-even}} \enspace + [\dots]_{\text{even-odd}} \enspace + [\dots]_{\text{even-even}} \enspace ,
\\
\widetilde{(\psi^2_{17}+\psi^2_{28}+\psi^2_{35}+\psi^2_{46}}&)=[\shortform{1\bar71'\bar7'}+\shortform{1\bar72'\bar8'}+\shortform{1\bar73'\bar5'}-\shortform{1\bar74'\bar6'}+\shortform{2\bar81'\bar7'}+\shortform{2\bar82'\bar8'}-\shortform{2\bar83'\bar5'}+\shortform{2\bar84'\bar6'}
\\
&+\shortform{3\bar51'\bar7'}-\shortform{3\bar52'\bar8'}+\shortform{3\bar53'\bar5'}+\shortform{3\bar54'\bar6'}-\shortform{4\bar61'\bar7'}+\shortform{4\bar62'\bar8'}+\shortform{4\bar63'\bar5'}+\shortform{4\bar64'\bar6'}]_{\text{odd-odd}} \enspace +
\\
&[\dots]_{\text{odd-even}} \enspace + [\dots]_{\text{even-odd}} \enspace + [\dots]_{\text{even-even}} \enspace ,
\\
\widetilde{(\psi^2_{18}+\psi^2_{27}+\psi^2_{36}+\psi^2_{45}}&)=[\shortform{1\bar81'\bar8'}-\shortform{1\bar82'\bar7'}+\shortform{1\bar83'\bar6'}+\shortform{1\bar84'\bar5'}-\shortform{2\bar71'\bar8'}+\shortform{2\bar72'\bar7'}+\shortform{2\bar73'\bar6'}+\shortform{2\bar74'\bar5'}
\\
&+\shortform{3\bar61'\bar8'}+\shortform{3\bar62'\bar7'}+\shortform{3\bar63'\bar6'}-\shortform{3\bar64'\bar5'}+\shortform{4\bar51'\bar8'}+\shortform{4\bar52'\bar7'}-\shortform{4\bar53'\bar6'}+\shortform{4\bar54'\bar5'}]_{\text{odd-odd}} \enspace +
\\
&[\dots]_{\text{odd-even}} \enspace + [\dots]_{\text{even-odd}} \enspace + [\dots]_{\text{even-even}} \enspace ,
\end{split}
\end{equation}
where the parentheses $[\dots]_{\text{odd-odd}}, [\dots]_{\text{odd-even}} , \dots$ are the $448=64 \times 7$ double products of terms that in Formulas \eqref{complex28} appear as summands both in an odd position, or the first in an odd and the second in an even  position, or similarly for the remaining two cases.

Next, look at $\mu_2 =  \sum_{1 \leq \gamma \leq 8} \psi^2_{\gamma 9}$ and at Formulas \eqref{complex8}. One gets in this ways a sum of $960 = 120 \times 8$ terms, among which $64 = 8 \times 8$ give:
\[
\mu'_2 = \frac{1}{2}(\shortform{1\bar11'\bar1'} + \shortform{1\bar12'\bar2'} + \dots +\shortform{1\bar18'\bar8'} + \dots + \shortform{8\bar8 8'\bar8'})
\]
and the remaining $896$ terms yield the component $\tilde \mu_2$ of
\[
\mu_2 = \mu'_2 +\tilde \mu_2.
\]
Now a computation shows that one half of the 896 terms in $\tilde \mu_2$ coincide, up to coefficient $-\frac{1}{2}$,  with the $448=64 \times 7$ terms of $\tilde \rho_2$. For example, among the following four terms:
\[
\tilde\mu_2 = \frac{1}{2}(\shortform{1\bar1'2'\bar2'} - \shortform{\underline{1\bar1'2' \bar2}}  -\shortform{\underline{1'\bar12'\bar2'}} + \shortform{1'\bar 1 2'\bar2} + \dots)
\]
only the two underlined correspond to terms in $\tilde \rho_2$. Thus, if $\underline{\tilde \mu_2}$ denotes  the sums of all underlined terms in $\tilde \mu_2$ we get:
\[
\tilde \rho_2 +\underline {\tilde \mu_2} = \frac{1}{2} \tilde \rho_2.
\]

Finally, look at terms of $\nu_2=  \sum_{1 \leq \delta \leq 9} \psi^2_{0 \delta}$ through Formulas \eqref{complex9}. This gives $1080=120 \times 9$ terms. Among them, $64 = 8 \times 8$ terms (coming from all the $\psi^2_{0 \delta}$ except $\psi^2_{0 9}$) yield the following sum
\[
\nu'_2 = \frac{1}{2}(\shortform{1\bar11'\bar1'} + \shortform{1\bar12'\bar2'} + \dots + \shortform{1\bar18'\bar8'} + \dots + \shortform{8\bar8 8'\bar8'}),
\]
so that
\[
\mu'_2 + \nu'_2 = \shortform{1\bar11'\bar1'} + \shortform{1\bar12'\bar2'} + \dots + \shortform{\bar18'\bar8'} + \dots + \shortform{8\bar8 8'\bar8'}.
\]
Moreover, we get:
\[
\psi^2_{0 9} = -\frac{1}{2}(\shortform{1\bar1 2\bar2} + \dots +  \shortform{7\bar 7 8\bar8}) -\frac{1}{2}(\shortform{1'\bar1' 2\bar2'} + \dots + \shortform{7'\bar7' 8'\bar8'}) + \frac{1}{2}(\shortform{1\bar11'\bar1'} + \shortform{1\bar1' 2' \bar2'} + \dots + \shortform{1\bar18'\bar8'}  + \dots + \shortform{8\bar8 8'\bar8'}),
\]
and by summing up:
\begin{equation}
\begin{split}
\rho_2 \vert_V +  \rho_2 \vert_{V'} +\mu'_2 + \nu'_2 + \psi^2_{0 9} = 2(\shortform{1\bar 1 2 \bar 2} + \dots + \shortform{7\bar 7 8 \bar 8}  + \shortform{1'\bar 1' 2' \bar 2'} + \dots + \shortform{7'\bar 7' 8' \bar 8'})
 + (\shortform{1\bar11'\bar1'} + \shortform{1\bar12'\bar2'} + \dots + \shortform{8\bar8 8'\bar8'}) \\ -\frac{1}{2}(\shortform{1\bar1 2\bar2} + \dots + \shortform{7\bar 7 8\bar8} + \shortform{1'\bar1' 2\bar2'} + \dots + \shortform{7'\bar7' 8'\bar8'})
 +\frac{1}{2}(\shortform{1\bar11'\bar1'}
 + \shortform{1\bar1' 2' \bar2'} + \dots +\shortform{1\bar18'\bar8'}  + \dots + \shortform{8\bar8 8'\bar8'})  = -3 \omega^2.
 \end{split}
\end{equation}

Next, we still have to deal with
\begin{equation}
\tilde \rho_2 +  \tilde \mu_2 + \tilde \nu_2  = \frac{1}{2} \tilde \rho_2 +  (\tilde \mu_2 - \underline{\tilde \mu_2}) + \tilde \nu_2 ,
\end{equation}
where the summands $\frac{1}{2} \tilde \rho_2 $, $\tilde \mu_2 - \underline{\tilde \mu_2}$,  $\tilde \nu_2$ consist respectively of 448, 448, 996 terms. A straightforward computation shows that this sum is zero. For example, the initial terms of the three sums are:
\[
\frac{1}{2} \tilde \rho_2 =  \frac{1}{2} [\shortform{1\bar21'\bar2'} - \shortform{1\bar22'\bar1'} -\shortform{2\bar1 1'\bar2'}+\shortform{2\bar1 2'\bar1'} + \dots]
\]
\[
\tilde \mu_2 - \underline{\tilde \mu_2} = \frac{1}{2} [\shortform{1\bar1' 2\bar2'} + \shortform{1'\bar12'\bar2} -\shortform{1\bar2' 2\bar1'}\shortform{2'\bar1 1'\bar2} + \dots]
\]
\[
\tilde \nu_2 = \frac{1}{2} [-\shortform{1\bar 1'2\bar2'} - \shortform{1\bar1'2'\bar 2} -\shortform{1'\bar1 2\bar2'}-\shortform{1'\bar1 2'\bar2'}
+\shortform{1\bar 2'2\bar1'} +\shortform{1\bar2'1'\bar 2} +\shortform{2'\bar1 2\bar1'}+\shortform{2'\bar1 1'\bar2} + \dots];
\]
and it is easy to see that the written terms cancel pairwise.
\end{proof}

It is now natural to give the following

\begin{de}\label{spin10}
Let $\tau_4$ be the fourth coefficient of the characteristic polynomial. We call the 8-form
\[
\Phi_{\Spin{10}} = \tau_4 (\psi^D)= \sum_{0\leq \alpha_1 < \alpha_2 <  \alpha_3 <  \alpha_4 \leq 9} ( \psi_{\alpha_1 \alpha_2} \wedge \psi_{\alpha_3  \alpha_4} - \psi_{\alpha_1 \alpha_3} \wedge \psi_{\alpha_2  \alpha_4} + \psi_{\alpha_1 \alpha_4} \wedge \psi_{\alpha_2  \alpha_3} )^2
\]
the \emph{canonical 8-form} associated with the standard $\Spin{10}$-structure in $\CC^{16}$.
\end{de}

\begin{re}\label{berger}
\rm{A close analogy appears between the constructions of the forms $\Phi_{\Spin{9}} \in \Lambda^8(\RR^{16})$ and $\Phi_{\Spin{10}} \in \Lambda^8(\CC^{16})$ (cf.\ Proposition \ref{acs:7->8->9} and Definition \ref{spin10}).

However, $\Phi_{\Spin{9}}$ can alternatively be defined by integrating the volume of octonionic lines in the octonionic plane. Namely, if $\nu_l$ denotes the volume form on the line $l = \{(x,mx)\}$ or $l = \{(0,y)\}$ in $\OO^2$, then
\[
\Phi_{\Spin{9}} =\int_{\OO P^1} p_l^*\nu_l \,dl,
\]
where $p_l: \OO^2 \cong \RR^{16} \rightarrow l$ is the orthogonal projection and $\OO P^1 \cong S^8$ is the octonionic projective line of all the lines $l \subset \OO^2$. This is the definition of $\Phi_{\Spin{9}}$ proposed by M. Berger in \cite{BerCCP}, somehow anticipating the spirit of calibrations.

Of course, an approach like this is not possible for $\Phi_{\Spin{10}}$, due to the lack of a similar Hopf fibration to refer to. Thus Definition \ref{spin10} appears to be a coherent algebraic analogy, and as we will see in next Section, it is suitable to represent a generator for the cohomology of the relevant symmetric space.}
\end{re}

\begin{re}\label{cliffordi}
\rm{Denote by $\mathfrak I$ the standard complex structure on $\CC^{16}$ and look at the ten endomorphisms $\mathfrak I, \I_1,\dots,\I_9$: the first of them is a complex structure and the remaining nine are involutions. The above discussion shows that these data are the right choice to give rise, via compositions of any pair of the ten endomorphisms, to the family $J^D = \{J_{\alpha \beta}\}_{0 \leq \alpha < \beta \leq 9}$, a basis of $\mathfrak{spin}(10)$. Note also that, on the fourth Severi variety $\EIII$, the complex structure $\mathfrak I$ can be looked at as element of the Lie algebra in the second factor of its holonomy $\Spin{10} \cdot \mathrm{U}(1)$.}
\end{re}

\section{The even Clifford structure, cohomology and proof of Theorem \ref{main}}\label{Rosenfeld}

In \cite{MoSCSR} the notion of \emph{even Clifford structure} on a Riemannian manifold $(M,g)$ is proposed as the datum of an oriented rank $r$ Euclidean vector bundle $E \rightarrow M$, together with a bundle morphism $\varphi : \Cl^0(E) \rightarrow \End{TM}$ from the even Clifford algebra bundle of $E$, and mapping $\Lambda^2 E$ into the skew-symmetric endomorphisms. Here $\Lambda^2 E$ is viewed as a sub-bundle of $\Cl^0(E)$ by the identification $e \wedge f \sim e \cdot f +h(e,f)$, for each $e,f \in E$ and where $h$ is the Euclidean metric on $E$.

Under the hypothesis of \emph{parallel} even Clifford structure (cf.\ \cite[page 945]{MoSCSR}), complete simply connected Riemannian manifolds admitting such a structure are classified (\cite[page 955]{MoSCSR}) and $\EIII$ turns out to be the only non-flat example with $r=10$.

\begin{pr}
The sub-bundle $E=<\I_0> \oplus <\I_1,\dots,\I_9> \subset \End{T\,\EIII}$ defines on $\EIII$ an even rank $10$ parallel Clifford structure.
\end{pr}

\begin{proof} The Clifford morphism $\varphi : \Cl^0(E) \rightarrow \End{T\,\EIII}$ is defined by composition. The property $\varphi(\Lambda^2 E) \subset \Endminus{T\,\EIII}$ is insured by $\I_\alpha \circ \I_\beta =-\I_\beta \circ \I_\alpha$ $(\alpha \neq \beta)$, just setting formally $\I_0 \circ \I_\beta =-\I_\beta \circ \I_0$, $\beta=1,\dots,9$, where $\I_0$ and $\I_\beta$ act on different factors of the prototype space $\CC \otimes \OO^2 \cong \CC^{16}$. Remind that the involutions $\I_1,\dots,\I_9$ are only defined locally, while the complex structure $\I_0$ is global on $\EIII$. The connection insuring the parallelism is the Levi-Civita connection on the endomorphisms' bundle.
\end{proof}

The (rational) cohomology of the fourth Severi variety $V_{(4)} \cong \EIII$ can be computed from the so-called A.~Borel presentation. Let $G$ be a compact connected Lie group, let $H$ be a closed connected subgroup of maximal rank and let $T$ be a common maximal torus. Then (cf.\ \cite[\S 26, page 19]{BorCEF}):
\begin{equation}\label{borel}
H^*(G/H) \cong H^*(BT)^{W(H)}/H^{>0}(BT)^{W(G)},
\end{equation}
where $BT$ is the classifying space of the torus $T$, $H^*(BT)^{W(H)}$ is the invariant sub-algebra of the Weyl group ${W(H)}$, and $H^{>0}(BT)^{W(G)}$ denotes the component in positive degree of the invariant sub-algebra of the Weyl group ${W(G)}$. One gets in this way the cohomology structure given by \eqref{integral}. Thus:


\begin{co}
The Poincar\'e polynomial, the Euler characteristics and the signature of $\EIII$ are given by
\begin{equation}
\begin{split}
&\mathrm{Poin}_{_{\EIII}} = 1+t^2+t^4+t^6+2t^8+2t^{10}+2t^{12}+2t^{14}+3t^{16}+ \dots ,
\\
& \qquad \qquad  \qquad \qquad   \chi_{_{\EIII} }=  27, \qquad \sigma_{_{\EIII}}= 3.
\end{split}
\end{equation}
\end{co}

Since $\EIII$ can be looked at as the projective plane over the complex octonions, it is natural to similarly construct a projective line over complex octonions, that turns out to be a totally geodesic submanifold of the former \cite{EscRGL,EscSSD}. This is the oriented Grassmannian:
\begin{equation}\label{line}
( \CC \otimes \OO) P^1= \mathrm{Gr_2} (\mathbb R^{10}),
\end{equation}
which is a non singular quadric $Q_8 \subset \CC P^9$, thus again a Hermitian symmetric space. Its rational cohomology is given by:

\begin{pr}
\begin{equation}\label{rosenfeld1}
H^*(\mathrm{Gr_2} (\mathbb R^{10})) \cong \ZZ[ e, e^\perp]/(\rho_{5},\rho_{8})
\end{equation}
where $e \in H^2$ and $e^\perp \in H^8$ are the Euler classes of the tautological vector bundle and of its orthogonal complement, and the relations are:
$\rho_5=ee^\perp \in H^{10}, \rho_8=e^8 - (e^\perp)^2 \in H^{16}$.
\end{pr}

Thus:

\begin{co}
The Poincar\'e polynomial, the Euler characteristics and signature are
\begin{equation}
\begin{split}
&\mathrm{Poin}_{_{\mathrm{Gr_2} (\mathbb R^{10})}} = 1+t^2+t^4+t^6+2t^8+ \dots ,
\\
&\qquad \chi_{_{\mathrm{Gr_2} (\mathbb R^{10})}} =  10  , \qquad \sigma_{_{\mathrm{Gr_2} (\mathbb R^{10})}} = 2.
\end{split}
\end{equation}
\end{co}

\emph{End of Proof of Theorem \ref{half-spin} and Proof of Theorem \ref{main}.}
To relate $\Phi_{\Spin{10}}$ with algebraic cycles in $V_{(4)} \subset \CC P^{26}$, look first at its totally geodesic projective line over complex octonions, i.e.\ at $Q_8 \cong Gr_2(\RR^{10})$, cf.\ \eqref{line}. There is a natural homogeneous sphere bundle one can consider over it, namely:
\begin{equation}\label{spherebundle}
\Spin{10}/\Spin{7} \times \mathrm{SO}(2) \stackrel{S^7}{\longrightarrow} Gr_2(\RR^{10}).
\end{equation}
This is in fact a sub-bundle to the restricted $S^9$-bundle
\[
\Esix/\Spin{9} \cdot \mathrm{U}(1)\vert_{Gr_2(\RR^{10})} \stackrel{S^9}{\longrightarrow} Gr_2(\RR^{10}) \subset \EIII,
\]
and these are the sphere bundles associated with the defining vector bundles of two even Clifford structures we are considering. The former is the rank 8 even Clifford structure defined on any oriented Grassmannian $\SO{k+8}/\SO{8}\times \SO{k}$ (cf.\ \cite[pages 955 and 965]{MoSCSR}), and here globally defined since $k=2$ is even. The latter is the rank 10 even Clifford structure we defined in the previous Section, and here restricted to $Gr_2(\RR^{10})$.

Next, from Formulas \eqref{complex9}, one sees that the vanishing of coordinates $\shortform{1',\dots , 8'}$ on the ``octonionic line" $Gr_2(\RR^{10})$ makes $\psi_{0\beta} \vert_{Gr_2(\RR^{10})}=0$ for $\beta = 1, \dots , 8$. Thus by looking at $\tau_4$ as sum of $4 \times 4$ principal minors, we obtain:
\[
\tau_4 \big(\psi^D\vert_{Gr_2(\RR^{10})} \big)= \tau_4\big(\psi_{\alpha \beta} \big)_{1 \leq \alpha < \beta \leq 8} + \sum_{1 \leq \alpha < \beta \leq 8} (\psi_{\alpha \beta}\psi_{09})^2.
\]

We need now the following fact. The K\"ahler 2-forms $\psi_{\alpha \beta}$ $(0 \leq \alpha < \beta \leq 9)$, that we wrote explicitly and globally on $\CC^{16}$, are of course only local on $\EIII$. They are associated with its non-flat even parallel rank $10$ Clifford structure. In situations like this it has been proved that such K\"ahler 2-forms turn out to be proportional to the curvature forms $\Omega_{\alpha \beta}$ of a metric connection on the structure bundle. An observation like this can be traced back to S. Ishihara \cite{IshQKM} in the context of quaternion-K\"ahler manifolds, where the local K\"ahler 2-forms  associated with the local compatible almost complex structures $I,J,K$ are recognized to be proportional to the curvature forms. Later, a similar argument has been developed by A. Moroianu and U. Semmelmann to get the same identification on Riemannian manifolds $M$ equipped with a non-flat parallel even Clifford structure \cite[Prop. 2.10 (ii) (a) at page 947 and Formula (14) at page 949]{MoSCSR}. In all these contexts, the Einstein property of the manifold is insured from the hypotheses. The proportionality stated in \cite{MoSCSR} reads:
\[
\Omega_{\alpha \beta} = \kappa \psi_{\alpha \beta}
\]
where $\kappa$ is deduced from the Ricci endomorphism $Ric$ on $M$ as follows
\[
Ric = \kappa(n/4 + 2r -4),
\]
and $n,r$ are the real dimension of the manifold and the rank of its non-flat even Clifford structure, respectively. Note that this insures that the 8-form $\tau_4 (\psi^D)$ on $\EIII$ is closed, ending the proof of Theorem \ref{half-spin}.

Coming now to Theorem \ref{main}, look at the totally geodesic $Q_8 \cong Gr_2(\RR^{10})$. On $Q_8$ the restriction of our $\psi_{\alpha \beta}$, $(1 \leq \alpha < \beta \leq 8)$ defines a rank 8 even Clifford structure. We normalize the metric on $\EIII$ in such a way that the induced metric on $Q_8$ is the same as the one induced by $Q_8 \subset \CC P^9$, with $\CC P^9$ of holomorphic sectional curvature $4$. This choice gives $Ric (\CC P^9)=20$ and $Ric(Q_8)=16$, so that the above identity with $n=16$ and $r=8$ gives $\kappa=1$. Therefore on the quadric $Q_8$:
\[
\Omega_{\alpha \beta} =  \psi_{\alpha \beta}.
\]

Thus the $\psi_{\alpha\beta}$ are local curvature forms of a metric connection on the rank 8 Euclidean vector bundle over $Gr_2(\RR^{10})$ having \eqref{spherebundle} as associated sphere bundle. This vector bundle is easily recognized to be $\gamma_2^\perp (\RR^{10})$, the orthogonal complement in $\RR^{10}$ of the tautological plane bundle over $Gr_2(\RR^{10})$, and it defines a non-flat even rank $8$ Clifford structure on $Gr_2(\RR^{10})$, cf.\ \cite[Tables at pages 955 and 965]{MoSCSR}. Using Chern-Weil theory, and recalling that $\psi_{09} \vert_{\CC^8} = \frac{i}{2} \big( \shortform{1\bar 1}+ \dots +\shortform{8 \bar 8} \big)$ and $\psi_{0\beta} \vert_{\CC^8} =0$ otherwise, we get:
\begin{equation}\label{last}
\tau_4 \big(\psi^D\vert_{Gr_2(\RR^{10})} \big)= (2\pi)^4 p_2(\gamma_2^\perp (\RR^{10})) -4 \omega^4.
\end{equation}
Here $p_2$ is the second Pontrjagin class and the last coefficient $-4$ comes from:
\[
\sum_{1 \leq \alpha < \beta \leq 8} (\psi_{\alpha \beta}\vert_{\CC^8})^2 = \rho_2\vert_V =-4(\omega\vert_V)^2,
\]
cf.\ beginning of the proof in \ref{tau}.

The equality \ref{last} can be reread by restricting at the maximal linear subspaces $\CC P^4$, $(\CC P^4)'$, that parametrize those oriented 2-planes in $\RR^{10}$ that are complex lines with respect to a complex structure preserving or reversing the orientation, cf.\ Remark \ref{subgrassmannians}. If $p_2=p_2(\gamma_1^\perp (\CC^{5}))$ is now the second Pontrjagin class of the orthogonal complement of the tautological line bundle over $\CC P^4$, this gives:
\[
\frac{1}{(2\pi)^4} \int_{\CC P^4 \; \text{or} \; (\CC P^4)'} \Phi_{\Spin{10}} = \int_{\CC P^4 \; \text{or} \; (\CC P^4)'} p_2 -\frac{4}{(2\pi)^4} \int_{\CC P^4 \; \text{or} \; (\CC P^4)'} \omega ^4.
\]
Here the last integral can be computed in terms of the Ricci form $\rho =5 \omega$, thus allowing to pass to the first Chern class $c_1=\frac{\rho}{2\pi}$ of $M= \CC P^4 \; \text{or} \; (\CC P^4)'$. Recalling that $c_1$ is 5 times a generator of the integral cohomology:
\[
\frac{1}{(2\pi)^4} \int_{M= \CC P^4 \; \text{or} \; (\CC P^4)'} \omega^4 =\frac{1}{(2\pi)^4} \frac{(2\pi)^4}{5^4} \int_M c_1^4(M) = 1.
\]
On the other hand the Pontrjagin class $p_2(\gamma_1^\perp (\CC^{5}))$ relates with the Chern classes of the same bundle as $p_2=2c_4-2c_1c_3+c_2^2$, giving again the generator of the top integral cohomology class. Thus, taking into account the orientation matter concerning $\CC P^4$ and $(\CC P^4)'$ and discussed in Remark \ref{subgrassmannians}, one gets:
\[
\frac{1}{(2\pi)^4} \int_{\CC P^4 \; \text{or} \; (\CC P^4)'} \Phi_{\Spin{10}} = \pm 1-4 = \begin{cases} -3  \\ -5 \end{cases}, \qquad \frac{1}{(2\pi)^4} \int_{\CC P^4 \; \text{or} \; (\CC P^4)'} \omega^4 = 1,
\]
and the conclusion follows. \hfill $\Box$

\emph{Acknowledgements.} This work originated from remarks on talks we gave in Moscow, Banff, Marburg, Bucharest. We thank in particular J. Berndt for pointing out possible relations of our previous work \cite{PaPSAC} with the Rosenfeld planes, U. Semmelman and A. Moroianu for taking the paper \cite{MoSCSR} to our attention and discussions on Clifford structures, and finally M. Radeschi for telling us about the notion of Clifford system.




\end{document}